\documentclass[12pt, reqno]{amsart}

\usepackage{amssymb, amsmath, amsthm,amsfonts}
\usepackage[backref]{hyperref}
\usepackage[alphabetic,backrefs,lite]{amsrefs}
\usepackage{amscd}  
\usepackage{fullpage}
\usepackage{stmaryrd}
\usepackage{mathrsfs}
\usepackage[all]{xy} 
\usepackage{multirow} 
\usepackage{array}
\usepackage{enumitem}

\usepackage[usenames,dvipsnames]{xcolor}
\usepackage{tikz}
\usepackage{cleveref}

\newtheorem{lemma}{Lemma}[section]

\newtheorem{prop}[lemma]{Proposition}
\newtheorem{cor}[lemma]{Corollary}
\newtheorem{conj}[lemma]{Conjecture}

\newtheorem{claim*}{Claim}
\newtheorem{thm}[lemma]{Theorem}
\newtheorem{defn}[lemma]{Definition}

\theoremstyle{definition}
\newtheorem{remark}[lemma]{Remark}

\newtheorem{rmk}[lemma]{Remark}

\newcommand{\eq}{\asymp}

\DeclareMathOperator{\codim}{codim}
\DeclareMathOperator{\dimrel}{dim_{rel}}
\DeclareMathOperator{\rg}{rk}
\DeclareMathOperator{\Res}{\mathrm{Res}_{E/\Q}}
\DeclareMathOperator{\Gal}{\mathrm{Gal}}
\DeclareMathOperator{\red}{\mathrm{red}}

\newcommand{\G}{\mathbb{G}}
\newcommand{\N}{\mathbb{N}}
\newcommand{\Q}{\mathbb{Q}}

\newcommand{\Z}{\mathbb{Z}}

\newcommand{\Ql}{\mathbb{Q}_{\ell}}
\newcommand{\Zl}{\mathbb{Z}_{\ell}}
\newcommand{\Fl}{\mathbb{F}_{\ell}}

\newcommand{\Flam}{\mathbb{F}_{\lambda}}
\newcommand{\Gm}{\mathbb{G}_m}

\DeclareMathOperator{\End}{End}
\DeclareMathOperator{\Endzero}{End^{\circ}}

\newcommand{\Hg}{{\rm{Hg(A)}}}
\newcommand{\MT}{{\rm{MT(A) }}}
\newcommand{\Hl}{{\rm{H_{\ell} }}}
\newcommand{\Gl}{{\rm{G_{\ell} }}}
\newcommand{\Glam}{{\rm{G_{\lambda} }}}

\newcommand{\Vl}{{\rm{V}}_{\ell}}
\newcommand{\Tl}{{\rm{T}}_{\ell}}

\newcommand{\Tlamcal}{\mathcal{T}_{\lambda}}
\newcommand{\Vlam}{{\rm{V}}_{\lambda}}
\newcommand{\Tlam}{{\rm{T}}_{\lambda}}

\newcommand{\Al}{{\rm{A}}[\ell]}
\newcommand{\Aln}{{\rm{A}}[\ell^n]}
\newcommand{\Alinf}{{\rm{A}}[\ell^{\infty}]}

\newcommand{\Oe}{\mathcal{O}_E}

\newcommand{\Ol}{\mathcal{O}_{E_{\ell}} }
\newcommand{\Elam}{E_{\lambda} }
\newcommand{\Olam}{\mathcal{O}_{\lambda} }

\newcommand{\Kmuinf}{K(\mu_{\ell^{\infty}})}

\newcommand{\Hlam}{H_{\lambda} }

\newcommand{\tlam}{t_{\lambda}}
\newcommand{\rlam}{r_{\lambda}}

\newcommand{\mult}{{\rm{mult}}}

\newcommand{\GL}{{\rm{GL}}}

\numberwithin{equation}{section}
\numberwithin{table}{section}

\title{Torsion for abelian varieties of type III}
\author{Victoria Cantoral Farf\'{a}n}

\address{Institut de Math\'ematiques de Jussieu - Paris Rive Gauche (IMJ-PRG)
UP7D - B\^atiment Sophie Germain - 75205 Paris, France \& The Abdus Salam International Center for Theoretical Physics, 11 Strada Costiera, 34151 Trieste, Italy}
\email{vcantora@ictp.it}
\urladdr{http://webusers.imj-prg.fr/~victoria.cantoral-farfan/}



\begin{document}

\maketitle

\today

\begin{abstract}

Let $A$ be an abelian variety defined over a number field $K$. The number of  torsion points that are rational over a finite extension $L$ is bounded polynomially in terms of the degree $[L:K]$ of $L$ over $K$. Under the following three conditions, we compute the optimal exponent for this bound in terms of the dimension of abelian subvarieties and their endomorphism rings: $(1)$ $A$ is geometrically isogenous to a product of simple abelian varieties of type I, II or III, according to the Albert classification; $(2)$ $A$ is of ``Lefschetz type'', that is, the Mumford--Tate group is the group of symplectic or orthogonal similitudes which commute with the endomorphism ring; $(3)$ $A$ satisfies the Mumford--Tate Conjecture. This result is unconditional for a product of simple abelian varieties of type I, II or III with specific relative dimensions. Further, building on work of Serre, Pink, Banaszak, Gajda and Kraso\'n, we also prove the Mumford--Tate Conjecture for a few new cases of abelian varieties of Lefschetz type.
\end{abstract}

\tableofcontents

\section{Introduction}\label{section 1}

Mordell-Weil's Theorem states that, for an abelian variety $A$ defined over a number field $K$, the group of $L$-rational points is finitely generated for any finite extension $L$ over $K$. 
It is of interest to know if we can obtain a bound for the number of torsion points that are rational over a finite extension $L$, that depends on the degree $[L:K]$ and the dimension of the abelian variety either when the abelian variety varies in a certain class or when the number field $L$ varies among all finite extensions of $K$. 
We will focus on the case where the abelian variety is fixed and the number field $L$ varies among all finite extensions of $K$; our main concern is to obtain an optimal bound which only depends on the degree of the number field $[L:K]$.  
Hindry and Ratazzi have given several results in this direction concerning some classes of abelian varieties (see \cite{HRce}, \cite{HR10} and \cite{HR}).
The aim of this paper is to present new results, which generalize the results of Hindry and Ratazzi. We consider the class of abelian varieties which are isogenous to a product of simple abelian varieties of type I, II or III (in the sense of Albert's classification) and such that each simple abelian variety is fully of Lefschetz type. This means that the Mumford--Tate Conjecture holds for each abelian variety and that the Mumford--Tate group is the group of symplectic or orthogonal similitudes which commute with the endomorphism ring. Recall that, when the abelian variety is simple of type III, it is known to carry exceptional Hodge classes, for further details we refer to \cite{Murty84}.

Let $A$ be an abelian variety defined over a number field $K$ of dimension $g$. Let $L$ be a finite extension of $K$. We would like to give an upper bound for the order of the finite group $A(L)_{\mathrm{tors}}$. In order to do so, let us define the following invariant, introduced by Ratazzi (see \cite{Nicolas}):

\begin{defn}The invariant $\gamma(A)$ is defined as follows:
$$
\gamma(A)= \inf \{ x>0 \,| \, \forall L/K \, \mathrm{finite}, \ |A(L)_{\mathrm{tors}}| \ll  [L:K]^x \},
$$
where the notation $\ll$ means that there exists a constant $C$ (which only depends on the abelian variety $A$ and $x$) such that 
$$|A(L)_{\mathrm{tors}}| \leq C [L:K]^x.$$
\end{defn}
More precisely $\gamma(A)$ is the optimal exponent: that means that it is the minimal value such that, for every positive $\varepsilon$, one has, for every finite extension $L/K$
$$
|A(L)_{\mathrm{tors}}| \ll  [L:K]^{\gamma(A)+\varepsilon}.
$$
Masser had already established an upper bound, which is polynomial on the degree of the number field $L$, see \cite{Masser}. More precisely, one has the the following bound:
$$
\gamma(A)\leq g.
$$
\begin{remark}
Masser's bound is optimal only in the case where $A$ is isogenous to a power of a CM elliptic curve.
\end{remark}
In order to present an optimal bound, Ratazzi gave, in the case of abelian varieties of CM type, an explicit formula for the invariant $\gamma(A)$ in terms of the characters of the Mumford--Tate group (see \cite[Th\'eor\`eme 1.10]{Nicolas}). Let us recall that the Mumford--Tate group and the Hodge group of an abelian variety $A$ are related by the following relation: 
$$\MT=\G_m\cdot \Hg.$$
Shortly after, Hindry and Ratazzi focused their attention on abelian varieties which are isogenous to a product of elliptic curves, they established an explicit formula for the invariant $\gamma(A)$ (see \cite[Th\'eor\`eme 1.6]{HRce}).
In 2010, the same authors gave an explicit formula of the invariant $\gamma(A)$ in the case where the abelian variety $A$ is isogenous to a product of simple abelian varieties of type $\rm{GSp}$, that means that they are simple abelian varieties such that the Mumford--Tate group is isogenous to the group of symplectic similitudes, and such that the Mumford--Tate Conjecture holds for these varieties (see \cite[Th\'eor\`eme 1.6]{HR10}). 
In the above reference, Hindry and Ratazzi stated the following conjecture:

\begin{conj}(\cite[Conjecture 1.1]{HR10})\label{conjecture HR}
Let $A$ be an abelian variety isogenous to a product of abelian varieties $\prod_{i=1}^d A_i^{n_i}$ where the $A_i$ are simple abelian varieties which are not pairwise isogenous. Then
$$\gamma(A)=\max_{\emptyset \neq I \subset \{1,...,d\}} \frac{2 \sum_{i\in I} n_i \dim A_i}{\dim \MT}.$$
\end{conj}
Conjecture \ref{conjecture HR} holds for abelian varieties of type I and II which are fully of Lefschetz type, see \cite[Th\'eor\`eme 1.6]{HR}, and for abelian varieties which are isogenous to a product of abelian varieties of type I or II (see \cite[Th\'eor\`eme 1.14]{HR}).

Let us introduce some notations concerning the abelian variety $A$. Denote by $D=\Endzero(A):=\End_{\overline{K}}(A)\otimes \Q$ the endomorphism algebra of $A$ and let $E=Z(D)$ be the center of $D$. Let us recall Albert's classification:
\begin{itemize}
\item \textbf{Type I:} $D=E$ is a totally real field of degree $e$, more precisely $g=eh$.

\item \textbf{Type II:} $D$ is a totally indefinite quaternion algebra over a totally real field $E$ of degree $e$, more precisely $g=2eh$.

\item \textbf{Type III:} $D$ is a totally definite quaternion algebra over a totally real field $E$ of degree $e$, more precisely $g=2eh$.

\item \textbf{Type IV:} $D$ is a division algebra over a CM field $E$.
\end{itemize}
For further information we refer the reader to \cite[Chap. IV]{Mum}. The integer $h$ introduced in each case corresponds to the following notation:
\begin{defn}
Let $A$ be an abelian variety of dimension $g$ of type I, II or III. Then the relative dimension of $A$ is the following integer:
$$
h:=\dimrel A=\left\{
\begin{aligned}
& \frac{g}{e} \quad \text{if $A$ is of type I,}\\
& \frac{g}{2e} \quad \text{if $A$ is of type II or III,}
\end{aligned}\right.
$$
where $e=[E:\Q]$ and $d^2=[D:E]$.
\end{defn}
Let $\phi$ be a polarization of the abelian variety $A$. It is known that the Hodge group $\Hg$ of $A$ is contained in $\mathrm{Sp}(D,\phi)$, where $\mathrm{Sp}(D,\phi)$ is the group of symplectic similitudes which commute with the endomorphisms in $D$ and preserve the polarization.
In the case where the abelian variety $A$ is of type I or II, we know that the Hodge group is always contained in $\Res\mathrm{Sp}_{2h}$. 
Moreover, when the abelian variety $A$ is of type III, we know that the Hodge group is always contained in $\Res\mathrm{SO}_{2h}$. This result follows from \cite{BGK} and further details will be found in section \ref{section 4}. In some cases, we have equality. For instance when
\begin{itemize}
\item $A$ is of type I or II: $\Hg=\Res\mathrm{Sp}_{2h}$ when $h=2$ or an odd integer (see \cite[Th\'eor\`eme 3.7]{HR} and \cite{BGKIetII}) and the Mumford--Tate Conjecture holds for $A$;
\item $A$ is of type III: $\Hg=\Res\mathrm{SO}_{2h}$ when $h$ is in the set $\{2k+1, k \in \N\}\setminus \left\{\frac{1}{2}\binom{2^{m+2}}{2^{m+1}}, m\in \N\right\}$ (see \cite{BGK} and Corollary \ref{corollaire A simple}) and the Mumford--Tate Conjecture holds for $A$.
\end{itemize}

Let $\ell$ be a prime number, we denote by $\Tl(A)= \varprojlim \Aln$ the $\ell$-adic Tate module of $A$. We consider the $\ell$-adic representation attached to the action of the absolute Galois group $G_K=\Gal(\overline{K}/K)$ over the $\ell^{\infty}$-torsion points:
$$
\rho_{\ell}:G_K \to \GL(\Tl(A)).
$$
Let $\Gl$ be the $\Ql$-algebraic group defined as the Zariski closure of the image of $\rho_{\ell}$, it is known as the $\ell$-adic monodromy group of $A$. Let $\Vl(A)=\Tl(A)\otimes \Ql$.
\begin{defn}
The abelian variety $A$ is fully of Lefschetz type if
\begin{itemize}
\item the inclusion $\MT\subset \mathcal{L}(A)$ is an equality,
\item the Mumford--Tate Conjecture hold for $A$, i.e. $\MT\otimes_{\Q} \Ql=\Gl$.
\end{itemize}
\end{defn}
Let us recall that the Lefschetz group $\mathcal{L}(A)$ of $A$ is the group of symplectic similitudes which commute with the endomorphisms sometimes denoted by $\mathrm{GSp}(D,\phi)$, for further details see \cite{MilneLC}. 
Moreover, from \cite[Theorem 4.3]{LP95} we know that, if the Mumford--Tate Conjecture holds for a prime number $\ell$, then it holds for every prime number.

In this paper, we prove Conjecture \ref{conjecture HR} in the case of abelian varieties of type III which are fully of Lefschetz type. 

\begin{thm}\label{theo1}
Let $A$ be a simple abelian variety defined over a number field $K$ of type III and dimension $g$. We have $e=[E:\Q]$, $d=2$ and $h$ is the relative dimension. We suppose that $A$ is fully of Lefschetz type. Then the invariant $\gamma(A)$ is equal to 
$$
\gamma(A)=\frac{2 d e h}{1 + eh(2h-1)}= \frac{2 \dim A}{\dim \rm \MT}.
$$ 
\end{thm}
The numerical value of $\gamma(A)$ is clearly sharper than the bound given by Masser, which would be $deh$. Note that the value we obtain for $\gamma(A)$ coincides with the conjectured value $\frac{2 \dim A}{\dim \rm \MT}$ and it will be useful to remark that the proof has a combinatorial component, which give to us the first equality. A posteriori, one can see that it corresponds with the conjectured value. 

Finally, we can generalize Conjecture \ref{conjecture HR} to the case where $A$ is isogenous to a product of simple abelian varieties of type I, II or III and fully of Lefschetz type. More precisely we present a new result where only type IV, in the sense of Albert's classification, is excluded.\\
\begin{thm}\label{theo3}
Let $A$ be an abelian variety defined over a number field $K$ isogenous over $\overline{K}$ to $\prod_{i=i}^{d}A_i^{n_i}$ where the abelian varieties $A_i$ are pairwise non isogenous over $\overline{K}$ of dimension $g_i$ and $n_i\geq 1$. We suppose that the abelian varieties $A_i$ are simple, not of type IV and fully of Lefschetz type. 
For every non empty subset $I\subset \{1,...,d\}$ we denote $A_I:=\prod_{i\in I} A_i$. Moreover we define $e_i=[E_i:\Q]$ where $E_i=Z(\Endzero(A_i))$, $h_i= \dimrel A_i$  and
$$
d_i=\left\{
\begin{aligned}
&1 \text{ if }A_i \text{ is of type I,}\\
&2 \text{ if }A_i \text{ is of type II or III,}
\end{aligned}\right.
\qquad
\eta_i=\left\{
\begin{aligned}
&1 \text{ if }A_i \text{ is of type I or II,}\\
&-1 \text{ if }A_i \text{ is of type III.}
\end{aligned}\right.
$$
Then one has:
$$
\gamma(A)=\max_{I} \left\{ \frac{2\sum_{i\in I} n_i d_i e_i h_i}{1 + \sum_{i\in I} e_i (2 h_i^2 +\eta_i h_i)}\right\}=\max_{I}\left\{ \frac{2\sum_{i\in I} n_i\dim A_i}{1 + \dim \rm Hg(\prod_{i\in I} A_i)}\right\}.
$$
\end{thm}
Recently, Zywina proved Conjecture \ref{conjecture HR} in the case of abelian varieties which verify the Mumford--Tate Conjecture (see \cite[Theorem 1.1]{David}). This result, proved independently, is more general than our theorem. Nevertheless let us remark that our method is more precise in the sense that it gives more information about the dimension of the stabilizers involved. Moreover, our method allows us to obtain a better lower bound for the degree of the field extension generated by a torsion point; this kind of information cannot be deduced from Zywina's result. 
\begin{thm}\label{theo4}
Let $A$ be a simple abelian variety defined over a number field $K$ of type III, with relative dimension $h$ and fully of Lefschetz type. Then there exists a constant $c_1:=c_1(A,K)>0$ such that, for every torsion point $P$ of order $m$ in $A(\overline{K})$, one has:
$$
[K(P):K]\geq c_1^{\omega(m)}m^{2h},
$$
where $\omega(m)$ is the number of prime factors that divide $m$.
\end{thm}
In the case of a product $A\simeq \prod_{i=1}^d A_i^{n_i}$ we obtain the same results with $h=\min_{i} h_i$.

Let us remark that one of the main hypotheses of Theorems \ref{theo1} and \ref{theo3} is that the abelian variety $A$ must be fully of Lefschetz type, in particular, that the Mumford--Tate Conjecture must hold for $A$. In this direction, the following two theorems give examples of abelian varieties that are fully of Lefschetz type.
The following theorem gives a correction of a subtle point of \cite[Theorem 5.11]{BGK} as has been pointed out by \cite[Remark 2.27]{LomHl}:

\begin{thm}\label{theo5}
Let $A$ be a simple abelian variety of dimension $g$ and of type III. Let us recall that $g=2eh$ where $e=[E:\Q]$ and $h$ is the relative dimension. We assume that  
$$
h\in \{2k+1, k \in \N\}\setminus \left\{\frac{1}{2}\binom{2^{m+2}}{2^{m+1}}, m\in \N\right\}.
$$
Then $A$ is fully of Lefschetz type.
\end{thm}
This theorem is stated in \cite{BGK} with the hypothesis that $h$ is odd. Nevertheless, it seems to be important to exclude the values of the form $\frac{1}{2}\binom{2^{m+2}}{2^{m+1}}$ with $m\in \N$ (see section \ref{section 5} for details).
Let us consider the following sets of numbers:
\begin{equation}\label{set sigma}
\Sigma:=\left\{g\geq 1; \exists k\geq 3, \exists a\geq 1, g=2^{k-1}a^k \right\} \bigcup \left\{g\geq 1; \exists k\geq 3, 2g=\binom{2k}{k}  \right\}.
\end{equation}

\begin{equation}\label{sigmaset}
\begin{aligned}
\Sigma':=\bigcup_{s>1} & \left\{ \left\{2^{(4k+3)s-1}, \, k\geq 0\right\} \cup \left\{2^{4ks-1}, \, k\geq 1\right\} \cup \left\{2^{2s(4k+1)-1}, \, k\geq 1\right\} \cup \left\{2^{2s-1}k^{2s}, \, k\geq 2\right\}\right.\\
&\left. \cup \left\{\frac{1}{2}\big(\binom{4k+4}{2k+2}\big)^s, \, k\geq 0\right\} \cup \left\{\frac{1}{2}\big(\binom{4k+2}{2k+1}\big)^{2s}, \, k\geq 0\right\}\right\}.
\end{aligned}
\end{equation}
We prove the following theorem which generalizes some results of \cite{Pink98}:

\begin{thm}\label{theo6}
Let $A$ be a simple abelian variety of type III such that the center of $\Endzero(A)$ is $\Q$. Assume that the relative dimension $h$ is not in the set $\Sigma'$. Then, the abelian variety is fully of Lefschetz type. In particular, the Mumford--Tate Conjecture holds for the abelian variety $A$.
\end{thm}

The following corollary gives a criterion for abelian varieties of type III to be fully of Lefschetz type in terms of the relative dimension:\\

\begin{cor}\label{corollaire A simple}
Let $E$ be a totally real number field of degree $e=[E:\Q]$. Let $A$ be an abelian variety defined over a number field such that $A$ is simple of type III and the center of the endomorphisms algebra of $A$ is $E$. Assume that one of the following hypotheses hold:

\begin{enumerate}
\item The relative dimension $h$ is in the set $\{2k+1, k \in \N\}\setminus \left\{\frac{1}{2}\binom{2^{m+2}}{2^{m+1}}, m\in \N\right\}$.

\item We have $e=1$ and the relative dimension $h$ is not in $\Sigma'$.
\end{enumerate}
Then $A$ if fully of Lefschetz type. In particular one has
$$
\gamma(A)=\frac{2dhe}{1+e(2h^2-h)}=\frac{2\dim\,A}{\dim \, \Hg+1}.
$$
\end{cor}

\begin{remark}
In the first point of the previous corollary we are excluding the following odd values for the integer $h$: 3, 35, 6435 (when $h\leq 10^6$). In the second case, when the abelian variety is simple of type III such that $E=\Q$, the previous corollary excludes the following $21$ possible values of $g=2h\leq 10^3$: 
$$
\begin{aligned}
\Sigma'=& \{4, 6, 8, 16, 36, 64, 70, 100, 128, 144, 196, 216, 256, 324, 400, 484, 512, 576, 676, 784, 900 \}.
\end{aligned}
$$
Note that when $g\leq 10^6$ the set $\Sigma'$ has $513$ elements.
\end{remark}
Using the previous results and some results of Ichikawa \cite{Ich} and Lombardo \cite{LomHl} one can prove that the product of simple abelian varieties of type I, II or III that are fully of Lefschetz type is also fully of Lefschetz type. We can therefore extend Corollary \ref{corollaire A simple} to the case of an abelian variety isogenous to a product of abelian varieties of type I, II or III.\\

\begin{cor}\label{corollaire A produit}Let $A$ be an abelian variety defined over a number field $K$ isogenous over $\overline{K}$ to $\prod_{i=i}^{d}A_i^{n_i}$ where the abelian varieties $A_i$ are pairwise non isogenous over $\overline{K}$ of dimension $g_i$ and $n_i\geq 1$. We suppose that the abelian varieties $A_i$ are simple and not of type IV. Assume that one of the following hypotheses hold:

\begin{enumerate}
\item The relative dimensions $h_i$ are 

\subitem - odd or equal to $2$ when the abelian variety $A_i$ is of type I or II, and 
\subitem - $h_i\in\left\{2k+1, k \in \N\right\}\setminus \left\{\frac{1}{2}\binom{2^{m+2}}{2^{m+1}}, m\in \N\right\}$ when $A_i$ is of type III.

\item We have $e_i=1$ (\textit{i.e.} $E_i=\Q$) and 
\subitem - $h_i$ is not in the set $\Sigma$ when $A_i$ is of type I or II, and 
\subitem - $h_i$ is not in the set $\Sigma'$ when $A_i$ is of type III.
\end{enumerate}
Then $A$ if fully of Lefschetz type. In particular, with the notations of Theorem \ref{theo3} one has

$$
\gamma(A)=\max_{I} \left\{ \frac{2\sum_{i\in I} n_i d_i e_i h_i}{1 + \sum_{i\in I} e_i (2 h_i^2 +\eta_i h_i)}\right\}=\max_{I}\left\{ \frac{2\sum_{i\in I} n_i\dim A_i}{1 + \dim \rm Hg(\prod_{i\in I} A_i)}\right\}.
$$ 
\end{cor}

This paper is organized as follows: in \cref{section 3} we present some lemmas of group theory. In \cref{vraie section 3} we describe $\lambda$-adic pairings coming from the Weil pairing as well as some properties of the $\ell$-adic Galois representations like the so called ``property $\mu$''. This property corresponds to the study of the cyclotomic part of those Galois representations. In \cref{section 4} we present the proofs of Theorems \ref{theo1} and \ref{theo3}. In \cref{cinq} we prove Theorem \ref{theo4}, and in the last section we present the proof of the new cases of the Mumford--Tate Conjecture in particular Theorems \ref{theo5} and \ref{theo6} and Corollaries \ref{corollaire A simple} and \ref{corollaire A produit}.


\section{Group lemmas}\label{section 3}

Let $V$ be a vector space endowed with a symplectic form (respectively a quadratic form $Q$), $G$ be the group of symplectic similitudes (respectively orthogonal), $\phi$ the bilinear form and $Q$ the quadratic form. The following calculation is perhaps already known but we have not found it in the literature.

\begin{thm}\label{th codimension}
Let $W$ be a vector subspace of $V$ of codimension $d$. Let us consider the stabilizer $G_W$ of $W$:
$$
G_W=\{g\in G, g_{|W}=id_W\}.
$$
Then
$$
\dim(G_W)=\frac{d(d+\varepsilon)}{2}, \quad \text{where} \quad \varepsilon=\left\{\begin{aligned}1 \quad &\text{in the symplectic case},\\ -1 \quad &\text{in the orthogonal case}.\end{aligned}\right.
$$
\end{thm}

\begin{remark}
Let us point out that what we call a stabilizer is called sometimes a fixator.
\end{remark}

\begin{proof}

We are going to prove this theorem in the orthogonal case. A similar computation of the codimension in the symplectic case can be found in  \cite[Paragraph 6, Remark 1]{HR}. Let $W$ be a vector subspace of $V$ of codimension $d$ and dimension $r$. We are going to split the proof into two parts. First, let us suppose that $W\cap W^{\perp} = \{0\}$, then 
$$
G_W=\left\{ \left( \begin{array}{c|c} 
I_r & 0 \\ \hline
0 & \rm{GO}_{d} 
\end{array}\right) \right\}.
$$
Therefore, $G_W \simeq \rm{GO}_{(W^{\perp}, \phi_{|W^{\perp}} )}$, and 
$$
\dim\, (G_W)=\frac{d(d-1)}{2}.
$$
Let us assume now that $W\cap W^{\perp} \neq \{0\}$, take $w_1 \in (W \cap W^{\perp})\setminus\{0\}$ and let $w_1^{'}\in  V\setminus W$ such that $\Pi:=<w_1,w_1^{'}>$ is a hyperbolic plane. One can prove by induction that $\dim\, (G_W)=\frac{d(d-1)}{2}$ using some well known results from algebraic geometry. Let $W^{'}=W + \Pi$; it is a subspace of codimension $d^{'}=d-1$, hence, by induction, we have that $\dim\, (G_{W^{'}})=\frac{(d^{'})(d^{'}-1)}{2}$. We introduce the following variety
$$
X:=\left\{y\in V, Q(y)=0 \text{ and } \, \phi(w,y)=\phi(w,w_1^{'}), \, \text{for all}\, w\in W \right\}.
$$
Let us consider the surjective map $f: G_W \to X$ defined by $f(g)=g(w_1^{'})$. The dimension of $G_W$ is equal to the dimension of $X$ plus the dimension of the generic fiber, that is:
$$
\dim\,  G_W = \dim \, X + \dim \, G_{W+\Pi}.
$$
Then using the Jacobian criterion we obtain
$$
\dim \, X= \dim \, T_y (X)= n - (r+1)=d-1.
$$
Finally
$$
\dim \, G_W = d-1 + \frac{(d-1)(d-2)}{2}=\frac{d(d-1)}{2}.
$$

\end{proof}

Let $H$ be a finite subgroup of $\Alinf$, then there exists an integer $n\in \N^{\ast}$, such that 
$$H\subset \Aln.$$ 
By the structure theorem of abelian $\ell$-groups we know that one can decompose $H$ in the following way:
\begin{equation}\label{decompositionH}
H=\prod_{i=1}^t (\Z/\ell^{m_{t-(i-1)}} \Z)^{\alpha_i}\subset \Aln,
\end{equation}
where $m_1<...<m_t$ is a strictly increasing sequence of integers, $\alpha_i$ are positive integers and $1\leq t\leq 2g$. 
For every $1\leq i\leq t$ we can consider the natural projection:
$$
\pi_{m_i}: \Tl(A) \twoheadrightarrow \Tl(A)/ \ell^{m_i}\Tl(A).
$$
Therefore we can associate to $H\subset \Aln$ a filtration of saturated submodules $W_t \subset ... \subset W_1$ of $\Tl(A)$ such that 
\begin{equation}\label{eq:filtration}
\pi_{m_i}(W_i)=H[\ell^{m_i}].
\end{equation}
For instance, using the previous decomposition \eqref{decompositionH} of $H$, we have:
\begin{equation}
\pi_{m_1}(W_1)=H[\ell^{m_1}]\simeq (\Z/\ell^{m_1}\Z)^{\alpha_1 + ... \alpha_t} \quad \text{and} \quad \pi_{m_t}(W_t)=H[\ell^{m_t}]\simeq (\Z/\ell^{m_t}\Z)^{\alpha_1}.
\end{equation}
We can associate to each submodule $W_i$ its stabilizer $G_{W_i}$ and we can describe it in the following way:
\begin{equation}
G_{W_i}=\{g\in \MT(\Zl), \, g(x)=x, \; \text{for all} \, x \in W_i\}.
\end{equation} 
We notice that $G_{W_1}\subset ... \subset G_{W_t}$.
Let us denote by $G(H)$ the stabilizer of $H$ and let us remark that $\rho_{\ell}(\Gal(\overline{K}/K(H)))$ can be identified with $G(H)$, where $K(H)$ is the extension generated by $H$. We can give the following description of $G(H)$ up to some finite index: 
\begin{equation}
G(H)=\{M\in \MT(\Zl), \, M\in G_{W_i} \mod \ell^{m_i}, \, \text{for }\, 1\leq i \leq t \}.
\end{equation}
This description of the stabilizer $G(H)$ allows us to introduce the following lemma. Actually, what it is important to remark is that the stabilizers $G_{W_i}$ are defined by some equations of bounded degree.

\begin{lemma}\label{lemme fusion}
Let $H$ be a finite subgroup of $\Alinf$, let $G$ be a subscheme over $\Z$ of $\rm GL_{2g,\Z}$, let $t$ be an integer and let $G_{W_1}\subset ... \subset G_{W_t}$ be $t$ subschemes over $\Zl$ of $G_{\Zl}$ defined as above of codimension $d_i=\codim G_{W_i}$.
For every prime number $\ell$ we have the following equality:
\begin{equation}\label{estimation}
(G(\Zl):G(H))\gg\ll_{A} \ell^{\sum_{i=1}^{t} d_i \big(m_i-m_{i-1}\big)},
\end{equation}
where $G(H)$ is the stabilizer of the subgroup $H$ and the notation $\gg\ll_{A}$ means that we have the equality modulo some multiplicative constant which depends on $A$.
\end{lemma}
First let us recall \cite[Th\'eor\`eme 1]{Oesterle}:
\begin{thm}[Oesterl\'e]\label{oesterle}
Let $Y\subset \Zl^N$ of dimension $r$ defined by some equations of degree less or equal than $d$. Then, for every $m\geq 1$ one has 
$$
|Y_m|\leq c(N,d,r) \ell^{mr},
$$
where $Y_m:=im(\red: Y \to (\Z/\ell^m \Z)^N)$ and $c(N,d,r)$ is a constant which only depends on $N$, $d$, and $r$.
\end{thm}

In order to prove Lemma \ref{lemme fusion}, we need the above result and the following lemma: 

\begin{lemma}\label{lisse-l-petit}
Let $H$ be a finite subgroup of $\Alinf$ and denote $Y=G(H)$. Then, for every integer $m \geq 1$ there exist constants $c_1$ and $c_2$, which depend on $\ell$, such that
$$
c_1 \ell^{m \dim Y} \leq |Y_m| \leq c_2 \ell^{m \dim Y}.
$$
\end{lemma}

\begin{proof}
We know that $Y=G(H)$ is defined by a finite number of polynomials of $N=(2g)^2$ variables with coefficients in $\Zl$. By Theorem \ref{oesterle} one can get an upper bound for $|Y_m|$. We know that for every integer $m\geq 1$ there exists a constant $c_2$, which depends on $N$, on the degree of polynomials which define $Y$ and on the dimension of $Y$ denoted $r$ such that:
\begin{equation}
|Y_m|\leq c_2 \ell^{m\dim G(H)}.
\end{equation}
Notice that the constant $c_2$ does not depend on the prime number $\ell$.

In order to get a lower bound of $|Y_m|$ we prove the following inequalities for every integer $m\geq 0$:

\begin{equation}\label{inequalities}
c(G,\ell) \ell^{m \dim G} \leq |G_m|\leq |Y_m| \times |(G/Y)_m|\leq c(\deg(Y),G,N)\ell^{m\dim(G/Y)}\times|Y_m|. 
\end{equation}
Then, we have
\begin{equation}
c_1(g,r,\ell)\ell^{m\dim Y}\leq |Y_m|,
\end{equation}
where $c_1(g,r,\ell)=\frac{c(G,\ell)}{c(\deg(Y),G,N)}$.

In order to obtain the first inequality in \eqref{inequalities} we use a weaker version of \cite[Theorem 9]{SerreApplications} which can be stated as follows:
\begin{equation}
\text{for all }\, m\geq 0,\qquad c(G,\ell) \ell^{m \dim G} \leq |G_m|
\end{equation}
where the constant $c(G,\ell)$ only depends on $G$ and on the prime number $\ell$.
Recall that for every finite subgroup $H\subset \Alinf$ there exists a filtration of submodules $W_t\subset ... \subset W_1$ of $\Tl(A)$ such that 
$$
G(H)=\{M\in \MT(\Zl), \, M\in G_{W_i} \mod \ell^{m_i},  \, \text{for }\, 1\leq i \leq t \},
$$
where the $G_{W_i}$ are the stabilizers of the submodules $W_i$.
Let $W=W_1 \otimes \Ql$ denote the $\Ql$-vector space of $\Vl(A)$ of dimension $r$. From Chevalley's Theorem (see \cite[Th\'eor\`eme p. 80]{Hum}) we know that there exists a vector space $U$ and a one-dimensional space $L\subset U$ such that there exists a rational representation $\rho: G\to \GL(U)$ such that $G(W)=Stab_L$.
As a corollary of Chevalley's Theorem we know that there exists an action of $G$ on the projective space $\mathbb{P}(U)$
$$
G\times \mathbb{P}(U) \to \mathbb{P}(U),
$$
such that the orbit of $[L]\in \mathbb{P}(U)$ can be identified with $G/G(W)$ because $G(W)=Stab_L$. 
Therefore the quotient $G/G(W)$ is a locally closed subvariety of $\mathbb{P}(U)$. Let us introduce the following map
\begin{equation}
\begin{aligned}
\psi:G & \to Z \subset \mathbb{P}(U)\\
g &\mapsto g\cdot[L]
\end{aligned}
\end{equation}
where $Z$ it the algebraic closure of $Z^{\circ}:=\psi(G)$. Hence we have the following isomorphism:
\begin{equation}\label{isomorphisme}
G/G(W)\simeq Z^{\circ}\subset Z.
\end{equation}
Let us remark that the variety $Z^{\circ}$ is actually a quasi-projective variety. Nevertheless, in order to use Oesterl\'e's Theorem one needs to work with affine varieties. Therefore, we can cover the variety $Z$ with some affine open subsets and then apply Oesterl\'e's Theorem to these open subsets.
Let $\ell$ be a prime number. Then, for every integer $m\geq 1$ the map $\psi$ induces a map
\begin{equation}
\psi_{\ell^m}:G_m \to Z^{\circ}_m.
\end{equation}
It follows from the isomorphism \eqref{isomorphisme} that
\begin{equation}
(G/G(W))_m \simeq Z^{\circ}_m.
\end{equation}

In order to prove the second inequality in \eqref{inequalities} we consider the following equality for every integer $m\geq 1$:

\begin{equation}\label{egalite}
|G_m|=\sum_{x\in Z^{\circ}_m} |\psi_{\ell^m}^{-1}\{x\}\cap G_m|.
\end{equation}
We know that $\psi^{-1}\{x\}\cap G_m$ is a homogeneous space under $G(W)$ and therefore $\psi_{\ell^m}^{-1}\{x\}\cap G_m$ is also a homogeneous space under $Y_m$. 
Consider the following surjection when $\psi_{\ell^m}^{-1}\{x\}\neq \emptyset$, 
\begin{equation}
\begin{aligned}
Y_m &\twoheadrightarrow \psi_{\ell^m}^{-1}\{x\}\cap G_m\\
g & \mapsto g\cdot y_0,
\end{aligned}
\end{equation}
where $y_0\in \psi_{\ell^m}^{-1}\{x\}\cap G_m$.
Hence
\begin{equation}\label{inegalite}
|\psi_{\ell^m}^{-1}\{x\}\cap G_m|\leq |Y_m|.
\end{equation}
From \eqref{egalite} and \eqref{inegalite} one deduces the inequality 
\begin{equation}
|G_m|\leq |Z^{\circ}_m|  \times |Y_m|,
\end{equation}
which is equivalent to the second inequality of \eqref{inequalities}. 

In order to obtain the last inequality of \eqref{inequalities} one needs to prove that, for every integer $m\geq 1$, we have 
\begin{equation}
|Z^{\circ}_m|\leq c(Y,G,N)\ell^{m \dim Z^{\circ}}.
\end{equation}
In order to do this we use Theorem \ref{oesterle} and the fact that $G/Y\simeq Z^{\circ} \subset \mathbb{P}(U)$ is entirely determined by certain equations of bounded degree.
Therefore, with the notations $Y=G(H)$ and $Y_m=im(Y\to \GL_{2g}(\Z/\ell^m \Z))$ we proved that, for every submodule $H\subset \Alinf$ we have,
$$
\text{for all }\, m\geq0, \qquad c_1(g,r,\ell)\ell^{m\dim Y}\leq |Y_m|\leq c_2(g,r,Y)\ell^{m\dim Y}.
$$
\end{proof}

Now we can proceed with the proof of Lemma \ref{lemme fusion}.
\begin{proof}
When the prime number $\ell$ is large enough, say $\ell\geq \ell_0(A,K)$, Lombardo proved that the stabilizer $G(H)$ is smooth over $\Zl$ (see \cite[Lemma 2.13]{LomHl}). Therefore one can use \cite[Lemma 2.4]{HR10} in order to obtain the following inequalities:
\begin{equation}
c_1\ell^{\sum_{i=1}^{t} d_i \big(m_i-m_{i-1}\big)} \leq (G(\Zl):G(H))\leq c_2 \ell^{\sum_{i=1}^{t} d_i \big(m_i-m_{i-1}\big)},
\end{equation}
where $c_1$ and $c_2$ are two constants independent of $m_i$ and $\ell$ for $\ell\geq \ell_0(A,K)$.

The main problem occurs when $\ell<\ell_0(A,K)$. In order to obtain our result we are going to introduce some notations from \cite{SerreApplications}.
Let $Y=G(H)$, then 
\begin{equation}
Y\subset G(\Zl)\subset \GL_{2g}(\Zl) \quad \text{and} \quad Y\subset \Zl^N \quad \text{where} \quad N=4g^2.
\end{equation}
Let $X_m:= (\Z/\ell^m \Z)^N$, we can define the following finite subgroup:
$$
Y_m:=im(\red: Y \to X_m).
$$
By Lemma \ref{lisse-l-petit} we know that, for $\ell< \ell_0(A,K)$, we have that for every integer $m\geq 1$ there exist two constants $c_1$ and $c_2$, which depend on $\ell$, such that:
$$
c_1 \ell^{m \dim Y} \leq |Y_m| \leq c_2 \ell^{m \dim Y}.
$$
Therefore there exists two constants $C_1$ and $C_2$, independent of the integers $m_i$, which possibly depend on $\ell$ for $\ell< \ell_0(A,K)$, such that
\begin{equation}
C_1\ell^{\sum_{i=1}^{t} d_i \big(m_i-m_{i-1}\big)} \leq (G(\Zl):G(H))\leq C_2 \ell^{\sum_{i=1}^{t} d_i \big(m_i-m_{i-1}\big)}.
\end{equation}

Finally our result is independent of the prime number $\ell$ and we obtain the following equality, up to some multiplicatives constants depending on $A$:
\begin{equation}
(G(\Zl):G(H))\gg\ll_{A} \ell^{\sum_{i=1}^{t} d_i \big(m_i-m_{i-1}\big)}.
\end{equation}
\end{proof}

\begin{remark}
As it has been pointed up by Serre \cite[Remarque p. 346]{SerreApplications}, one can use the methods, developed by Oesterl\'e and Robba (see Theorem \ref{oesterle}), in order to obtain a better estimation of the index $(G(\Zl):G(H))$ than the one given by \cite[Th\'eor\`eme 8]{SerreApplications}. 
\end{remark}


\section{Galois representations}\label{vraie section 3}

\subsection{Weil pairing}\label{BGK}

Let $A^{\vee}$ be the dual variety of $A$, then there exists a bilinear non-degenerate form over $\Tl(A)\times \Tl(A^{\vee})$ which is Galois equivariant, called the Weil pairing:
$$
<\cdot,\cdot>: \Tl(A)\times \Tl(A^{\vee}) \to \Zl(1)=\varprojlim \mu_{\ell^n},
$$
we refer the reader to \cite[Chap IV]{Mum} for further information.
Let $\phi:A\to A^{\vee}$ be a polarization of $A$. It induces a non-degenerate, alternating, bilinear pairing:
$$
\phi_{\ell^{\infty}}: \Tl(A)\times \Tl(A) \buildrel{id\times \phi}\over{\to} \Tl(A)\times \Tl(A^{\vee}) \buildrel{<\cdot,\cdot>}\over{\to} \Zl(1)=\varprojlim \mu_{\ell^n}.
$$

Recall that $D=\Endzero(A)$ is a quaternion algebra defined over the totally real field $E$. We define the following set:
\begin{defn}\label{ensemble S}
Let $\mathcal{S}$ be the finite set of prime numbers $\ell$ such that $\ell$ is ramified in $\mathcal{O}_E$, or $\ell$ divides the degree of the fixed polarization $\phi$ of $A$, or in the case of type II or III, the quaternion algebra $D$ does not decompose at some $\lambda| \ell$. 
\end{defn}

Let us recall that $\Ol=\prod_{\lambda|\ell} \Olam$, hence, 
$$\Tl(A)=\prod_{\lambda|\ell} \Tl(A)\otimes_{\Zl}\Olam,$$
where by definition $\Tlamcal(A):=\Tl(A)\otimes_{\Zl}\Olam$. 
By \cite[Theorem 3.23]{BGK} we know that, in the case where $A$ is an abelian variety of type II or III and $\ell \notin \mathcal{S}$, $\Tlamcal(A)=\Tlam(A)\oplus \Tlam(A)$ where $\Tlam(A)$ is a free $\Olam$-module of rank $2h$. Moreover, for every $\lambda|\ell$ there exists a symmetric, non-degenerate pairing 
$$\phi_{\lambda^{\infty}}: \Tlam(A) \times \Tlam(A) \to \Olam$$ 
compatible with the $G_K$-action which is induced by $\phi_{\ell^{\infty}}$. 
Therefore the $\Ql$-vector space $\Vl$ decomposes: 
$$\Vl=\prod_{\lambda|\ell} \Vlam \oplus \Vlam \quad \text{and} \quad \Hg\otimes \Ql\subset \mathrm{Sp}(D,\phi)\otimes \Ql=\prod_{\lambda|\ell} \mathrm{SO}(\Vlam, \phi_{\lambda}).$$

\subsection{Property $\mu$}\label{propriete mu}

Let us introduce the property called ``propri\'et\'e $\mu$'' (see \cite[Definition 6.3]{HR10}). We say that the abelian variety $A$ satisfies the property $\mu$ if for every prime number $\ell$ and every subgroup $H\subset \Alinf$ there exists an integer $m=m(H)$, such that, up to some finite index bounded independently of $\ell$ (notation $\eq$), one has:
$$
K(H)\cap \Kmuinf)\eq K(\mu_{\ell^m}).
$$
Let assume that the abelian variety satisfies the Mumford-Tate Conjecture. 
Then in \cref{diagramme2} we have the following equalities up to some finite indexes bounded independently of $\ell$:
$$
\Gal(K(\Alinf)/K)\eq \MT(\Zl),
$$
$$
\Gal(K(\Alinf)/\Kmuinf)\eq \Hg(\Zl),
$$
$$
\Gal(\Kmuinf/K)\eq \Gm(\Zl).
$$
\begin{figure}[h!]
\begin{center}
\begin{tikzpicture}
\node (1) at (0,6) {$K(\Alinf)$};
\node (2) at (-3.5,4) {$K(H)$};
\node (3) at (3.5,4) {$\Kmuinf$};
\node (4) at (0,2) {$K(\mu_{\ell^m})$};
\node (5) at (0,0) {$K$};

\node (6) at (-4.9,1) {$\MT(\Zl)$};
\node (7) at (3.2,1.8) {$\G_m(\Zl)=\Zl^{\times}$};
\node (8) at (3,5.2) {$\Hg(\Zl)$};

\draw (1)--(2); \draw (1)--(3);
\draw (4)--(2); \draw (4)--(3);
\draw (5)--(2); \draw (5)--(3);
\draw (4)--(5);

\draw[dashed] (5) edge[bend left=22] (-3.5,1) ;
\draw[dashed] (-3.5,1) edge[bend left=45] (-3.5,5.3);
\draw[dashed] (-3.5,5.3) edge[bend left=20] (1) ;

\end{tikzpicture}
\end{center}
\caption{General case : $H\subset \Alinf$}\label{diagramme2}
\end{figure}
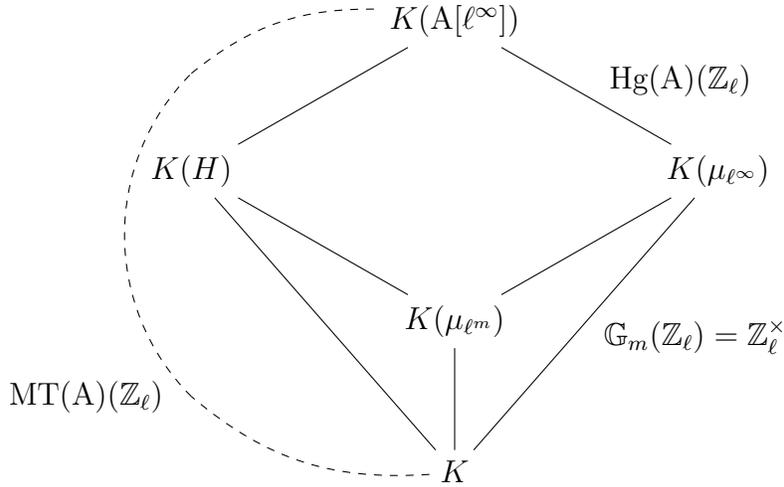\\
Let us introduce the following groups:
$$
G_0(H):=\{\sigma \in \MT(\Zl), \sigma_{|H}=id_{|H}\}=\Gal(K(\Alinf)/K(H)),
$$
$$
G(H):=G_0(H)\cap \Hg(\Zl).
$$
\begin{lemma}\label{lemma delta}
Let $H\subset \Alinf$ and let $\delta(H):=[K(\mu_{\ell^m}):K]$. Then we have:
$$
[K(H):K]=(\Hg(\Zl):G(H))\cdot \delta(H),
$$
up to some finite index bounded independently of the prime number $\ell$.
\end{lemma}

\begin{proof}
Recall that $G(H)$ is the stabilizer of $H$ in $\MT(\Zl)$ and that the abelian variety $A$ is simple of type III. The key point of the proof is that the morphism 
$$\mult: G(H)\to \Gm$$ 
is surjective for every maximal isotropic subgroup $H$ of $\Alinf$. One can see that this is a re adaptation of \cite[Proposition 5.5]{HR}. Then we will have $\mult(G(H))(\Zl)=\Zl^{\times}$ and therefore $\delta(H)=[K(\mu_{\ell^m}):K]=(\Zl^{\times}:\mult(G(H)(\Zl))=1$.
\end{proof}

Let us remark that in the case when $H\subset \Aln$ and the property $\mu$ holds, there exists an integer $m=m(H)$ such that, up to some finite index bounded independently of $\ell$, one has:
\begin{equation}\label{equation cas simple}
K(H)\cap K(\mu_{\ell^n})=K(H)\cap \Kmuinf \eq K(\mu_{\ell^m}).
\end{equation}


\section{Main results and proofs}\label{section 4}

In this section we are going to present a complete proof of the following theorem:

\begin{thm}\label{theo1 section 4}
Let $A$ be a simple abelian variety of type III, defined over a number field $K$ and dimension $g$. We have $e=[E:\Q]$, $d=2$ and $h$ the relative dimension. We suppose that $A$ is fully of Lefschetz type.
Then the invariant $\gamma(A)$ is equal to 
$$
\gamma(A)=\frac{2 d e h}{1 + eh(2h-1)}= \frac{2 \dim A}{\dim \rm \MT}.
$$ 
\end{thm}

Let us recall the definition of the invariant $\gamma(A)$:
\begin{equation}
\gamma(A)=\inf \{ x>0 \,| \, \forall L/K \, \mathrm{finite}, \,  |A(L)_{\mathrm{tors}}| \ll [L:K]^x \}.
\end{equation}
The main result needed to compute the invariant $\gamma(A)$ is the criterion of the independence of $\ell$-adic representations introduced by Serre in \cite{Serre}. Roughly  speaking, it says that there exists a finite extension $K'/K$, such that for every finite subgroup $H_{\mathrm{tors}}$ of $A(L)_{\mathrm{tors}}$ such that, if we write
$$
H_{\mathrm{tors}}=\prod_{\ell\in \mathcal{P}} H_{\ell},
$$
where $H_{\ell}$ is a subgroup of $\Alinf$. Then we have
$$
[K'(H_{\mathrm{tors}}):K']= \prod_{\ell\in \mathcal{P}} [K'(H_{\ell}):K'].
$$ 
Using this criterion, one can concentrate attention at finite subgroups $H$ of $\Alinf$. Therefore, we suppose that the number field $K$ is such that the $\ell$-adic representations are independent. We are going to work with the following definition of the invariant $\gamma(A)$:
\begin{equation}
\gamma(A)=\inf \{ x>0 \,| \, \forall H\subset \Alinf, \,  |H| \ll [K(H):K]^x \}.
\end{equation}
In order to compute the invariant $\gamma(A)$, we use the following equivalence: for every finite subgroup $H$ of $\Alinf$ we have
$$
|H| \ll [K(H):K]^{\gamma(A)} \quad \iff \quad \gamma(A) \geq \frac{\log_{\ell} |H|}{\log_{\ell} [K(H):K]}.
$$
Note that the determination of $\gamma(A)$ depends only on the order of the finite subgroup $H$ and the degree of the extension $K(H)$ over $K$.

After extending $K$ and replacing $A$ by an isogenous abelian variety, we may assume:
\begin{itemize}
\item the abelian variety $A/K$ is such that $\End_{\overline{K}}(A)=\End_K(A)$ which we denote by $\End(A)$,
\item we suppose that $\End(A)$ is a maximal order of $\End(A)\otimes \Q$,
\item the subgroup $H\subset \Alinf$ is stable under the action of $\End(A)$,
\item the number field $K$ is such that $\Gl$ is connected and the $\ell$-adic representations are independent in the sense of Serre (see \cite{Serre}).
\end{itemize}
Recall that the notation $\eq$ means that it is an equality up to some finite index bounded independently of the prime number $\ell$.

\begin{proof}

Recall that $A$ is a simple abelian variety defined over a number field $K$ of type III, dimension $g=2eh$. Let $G_K$ be the absolute Galois group and let $\ell$ be a prime number not in $\mathcal{S}$. Moreover we assume that the abelian variety $A$ is fully of Lefschetz type, more precisely one has the following equality
\begin{equation}
\Hg_{\Ql}=\prod_{\lambda|\ell}\mathrm{SO}_{2h}(\Olam)=\prod_{\lambda|\ell}\mathrm{Res}_{\Elam/\Ql}\mathrm{SO}_{2h,\Elam}.
\end{equation}
We can therefore deduce that 
$$
\dim \Hg_{\Ql}= \sum_{\lambda|\ell}[\Elam:\Ql]\frac{2h(2h-1)}{2}=e(2h^2-h).
$$

The proof of this theorem is going to be divided into two parts. 

\textbf{Part 1.} Assume that $H \subset \Al$, then one has the following decomposition for $\Al=\Tl(A)/\ell \Tl(A)$:
$$\Al=\prod_{\lambda|\ell} \Tlam[\ell]\oplus \Tlam[\ell],$$ where $\Tlam[\ell]=\Tlam(A)/\ell\Tlam(A)$.
Thus, for every subgroup $H$ of $\Al$ one has the following decomposition 
$$
H=\prod_{\lambda|\ell} \Hlam \oplus \Hlam,
$$
where $\Hlam$ is a subgroup of $\Tlam[\ell]$ of dimension $\rlam\leq 2h$.
Hence, one can deduce the order of $H$:
\begin{equation}
|H|=\ell^{2\cdot \sum_{\lambda|\ell} f(\lambda)\rlam},
\end{equation}
where $f(\lambda)$ is the residual degree in $\lambda$. 

Let us determine the degree of the extension $K(H)$ over $K$. In order to accomplish this, we use the property $\mu$ (see \cref{propriete mu}). We know that there exists an integer $m=m(H)$, such that, up to some finite index bounded independently of $\ell$, one has, as in \eqref{equation cas simple}:
$$
K(H)\cap \Kmuinf\eq K(\mu_{\ell^m}).
$$
In particular, because $H\subset A[\ell]$, one has
$$
K(H)\cap \Kmuinf\eq \left\{ 
\begin{array}{l}
K\\
K(\mu_{\ell}).
\end{array}
\right.
$$
Let us introduce the following groups as in section \ref{vraie section 3}:
$$
G_0(H):=\{\sigma \in \MT(\Fl), \sigma_{|H}=id_{|H}\}=\Gal(K(A[\ell])/K(H)),
$$
$$
G(H):=G_0(H)\cap \Hg(\Fl).
$$
By Lemma \ref{lemma delta} we know that $\delta(H):=[K(\mu_{\ell^m}):K]$ and, up to some finite index bounded independently of $\ell$, we have the following equality, 
\begin{equation}\label{degree}
[K(H):K]=[K(H):K(\mu_{\ell^m})][K(\mu_{\ell^m}):K]=(\Hg(\Fl):G(H))\cdot \delta(H).
\end{equation}
In order to compute $[K(H):K]$ one needs to determine $(\Hg(\Fl):G(H))$. We use Theorem \ref{th codimension} and Lemma \ref{lemme fusion} to achieve this.
Let us recall that for every subgroup $H\subset \Al$ one has the following decomposition: $$H=\prod_{\lambda|\ell}\Hlam \oplus \Hlam.$$ 
Consequently, one can apply the property $\mu$ to each subgroup $\Hlam \subset \Tlam[\ell]$ of dimension $\rlam$. Thus, there exists an integer $m_{\lambda}$ such that 
$$
K(\Hlam)\cap \Kmuinf\eq K(\mu_{\ell^{m_{\lambda}}}).
$$
We will use the notation $m:=\max_{\lambda|\ell}m_{\lambda}$. We know that, up to some finite index bounded independently of the prime number $\ell$:
$$
\Gal(K(\Al)/K(\mu_{\ell}))\eq \prod_{\lambda|\ell} \Gal(K(\Tlam[\ell])/K(\mu_{\ell})).
$$
Thus one has:
$$
(\Hg(\Fl):G(H))=\prod_{\lambda|\ell}(\Hg(\Flam):G(\Hlam)).
$$
In order to compute $(\Hg(\Fl):G(H))$ one needs to compute $(\Hg(\Flam):G(\Hlam))$ for every $\lambda|\ell$. Using Lemma \ref{lemme fusion} one obtains:
$$
(\Hg(\Flam):G(\Hlam))\gg\ll (\#\Flam)^{\codim G(\Hlam)},
$$
therefore
$$
(\Hg(\Fl):G(H))\gg\ll \prod_{\lambda| \ell} (\sharp \Flam)^{\codim \, G(\Hlam)}.
$$
Let $\delta$ be an integer such that $\delta(H)=\ell^{\delta}$. Then the equation \eqref{degree} becomes
\begin{equation}
[K(H):K]\gg\ll \ell^{\sum_{\lambda|\ell} f(\lambda) \codim \, G(\Hlam) + \delta},
\end{equation}
where, by Theorem \ref{th codimension}, one has $$\codim\, G(\Hlam)=\frac{(4h-1)\rlam - \rlam^2}{2}.$$

We know that, for every subgroup $H\subset \Alinf$, one has
$$
\gamma(A) \geq \frac{\log_{\ell} |H|}{\log_{\ell} [K(H):K]}.
$$
Hence
$$
\gamma(A)= \max_{\rlam} \psi(\underline{r}),
$$
where $\underline{r}=(\rlam)_{\lambda | \ell}$ and the function $\psi$ is defined as follows:

$$\psi(\underline{r})=\frac{2\sum_{\lambda | \ell} f(\lambda) \rlam}{\delta+ \sum_{\lambda | \ell} f(\lambda) \codim G_{\Hlam} }.$$

The idea now is to study the maximum of this function when the integers $\rlam$ vary. We are going to separate the investigation of the function $\psi$ into two parts: 
\begin{enumerate}
\item First of all, we suppose that $\Hlam$ is contained in a maximal isotropic space, hence $0 \leq \rlam\leq h$ for every $\lambda | \ell$ and $\delta(\Hlam)=1$ and so $\delta=0$. 
\item Secondly we suppose that $\Hlam$ is not contained in a maximal isotropic space then $\delta(\Hlam)=\ell$ and therefore $\delta=1$. 
\end{enumerate}
Then, one can obtain the maximum of the function $\psi$ when one compares the maximum obtained in each case. We are going to study the following function:
$$\psi(\underline{r})=\frac{2\sum_{\lambda | \ell} f(\lambda)\rlam}{\delta+ \sum_{\lambda | \ell} f(\lambda) \frac{(4h-1)\rlam-\rlam^2}{2}}.$$

In the case, when $\delta=0$ we have:
$$\psi(\underline{r})=\frac{4\sum_{\lambda | \ell} f(\lambda)\rlam}{\sum_{\lambda | \ell} f(\lambda)(-\rlam^2 + \rlam(4h-1))}.$$
We observe that $\psi$ is an increasing function and its maximum is reached when $\rlam=h$, for all $\lambda|\ell$. Then we have: 
$$
\max_{\rlam \in [0,h]} \psi (\underline{r})= \frac{4}{3h-1},
$$
where the maximum is over all $\underline{r}$ with $\rlam\in[0,h]$, for all $\lambda|\ell$.

In the case, when $\delta=1$ we have:
$$\psi(\underline{r})=\frac{4\sum_{\lambda | \ell} f(\lambda)\rlam}{2 + \sum_{\lambda | \ell} f(\lambda)(\rlam(4h-1)-\rlam^2) }.$$
Again, $\psi$ is an increasing function, however in this case its maximum is reached when $\rlam=2h$, for all $\lambda|\ell$. Then we have: 
$$
\max_{\rlam \in [0,2h]} \psi(\underline{r})= \frac{4eh}{1+e(2h^2 - h)}.
$$
where the maximum is over all $\underline{r}$ with $\rlam\in[0,2h]$, for all $\lambda|\ell$.

We can therefore conclude that the maximum of the function $\psi$ is
$$
\max_{\rlam\in [0,2h]} \psi(\underline{r}) = \max \bigg(  \frac{4}{3h-1},\frac{4eh}{1+e(2h^2 - h)}\bigg)=\frac{4eh}{1+e(2h^2 - h)}.
$$
Since $d=2$ one has
$$
\gamma(A)=\frac{2deh}{1+e(2h^2 - h)}=\frac{2\dim A}{1+ \dim \rm Res_{E/\Q} \rm{SO}_{2h}}=\frac{2\dim A}{\dim \MT},
$$
as expected.

\textbf{Part 2.} Assume now that  $H\subset \Alinf$. This means that there exists $n\in \N^{\ast}$ such that $H\subset \Aln$. As in the previous case, $H=\prod_{\lambda|\ell} \Hlam \oplus \Hlam$ where $\Hlam\subset \Tlam[\ell^n]$. One has the following decomposition as in \eqref{decompositionH}:
\begin{equation}
\Hlam = \prod_{i=1}^{t_{\lambda}} (\Z/\ell^{m_{\lambda}^i}\Z)^{\alpha_{\lambda,i} f(\lambda)}\simeq \prod_{i=1}^{\tlam} (\Olam/\ell^{m_{\lambda}^i}\Olam)^{\alpha_{\lambda,i}},
\end{equation}
where $1\leq \tlam \leq 2h$, $\alpha_{\lambda,i}$ are integers and $m_{\lambda}^{t_{\lambda}}<...<m_{\lambda}^1$ is a strictly decreasing sequence of integers. 
Therefore one can deduce that the order of $\Hlam$ is $|\Hlam|=\ell^{f(\lambda)\sum_{i=1}^{\tlam} \alpha_{\lambda,i}m_{\lambda}^i}$ and that the order of the subgroup $H$ is
\begin{equation}
|H|=\ell^{\sum_{\lambda|\ell} 2 f(\lambda) \sum_{i=1}^{\tlam} \alpha_{\lambda,i} m_{\lambda}^i}.
\end{equation}

Let us determine the degree of the extension $K(H)$ over $K$. As before, by Lemma \ref{lemma delta}, we know that $\delta(H):=[K(\mu_{\ell^m}):K]$ and we have the following equality, up to some finite index bounded independently of $\ell$, 
\begin{equation}\label{degree bis}
[K(H):K]=(\Hg(\Zl):G(H))\cdot \delta(H).
\end{equation}
As before, one has
\begin{equation}\label{h dans le cas H dans Alinf}
(\Hg(\Zl):G(H))=\prod_{\lambda| \ell} (\Hg(\Olam):G(\Hlam)).
\end{equation}
Our goal is to estimate the value of
\begin{equation}
(\Hg(\Olam):G(\Hlam)),
\end{equation}
for every $\lambda|\ell$.
In order to do so, we use Lemma \ref{lemme fusion} and we study more deeply the structure of each stabilizer $G(\Hlam)$.
As in section \ref{section 3}, we introduce, for every $\lambda|\ell$ and every $1\leq i \leq t_{\lambda}$, a filtration $W_{\tlam}\subset ... \subset W_1$ of saturated submodules of $\Tlam$ associated to the subgroups $\Hlam$ (see \eqref{eq:filtration}).
We define 
$$
r_{i}:=\rg_{\Olam} \,W_{i} = \sum_{k=1}^{t_{\lambda} - (i-1)} \alpha_{\lambda,k}.
$$
Let $G_{W_{i}}$ be the stabilizer of $W_{i}$, and hence $G(\Hlam)$ can be described as follows:
$$
G(\Hlam)=\{M\in \MT(\Olam), \, M \in G_{W_{i}} \mod \, \ell^{m_{\lambda}^{t_{\lambda} - (i-1)}} \, \text{for} \, 1\leq i \leq t_{\lambda}\}.
$$
By Theorem \ref{th codimension} we know that the codimension $d_i$ of $G_{W_{i}}$ is: 
$$
d_{i}=\frac{r_{i}(4h -1 - r_{i})}{2}.
$$
By Lemma \ref{lemme fusion} we know that for every prime number $\ell$, we have the following equality, up to some constants
$$
(\Hg(\Olam):G(\Hlam)) \gg\ll \ell^{ f(\lambda) \sum_{i=1}^{t_{\lambda}} d_{i}(m_{\lambda}^{t_{\lambda}-(i-1)}-m_{\lambda}^{t_{\lambda}-(i-1)+1}) }.
$$
Consequently equality \eqref{h dans le cas H dans Alinf} becomes:
\begin{equation}\label{hodge dans le cas H dans Alinf}
(\Hg(\Zl):G(H))\gg\ll \ell^{\sum_{\lambda|\ell} f(\lambda) \sum_{i=1}^{t_{\lambda}} d_{i}(m_{\lambda}^{t_{\lambda}-(i-1)}-m_{\lambda}^{t_{\lambda}-(i-1)+1}) }.
\end{equation}
By convention we have put $m_{\lambda}^{t_{\lambda}+1}=0$.
This means that the equality \eqref{degree bis} becomes
$$
[K(H):K]\gg\ll \ell^{\sum_{\lambda|\ell} f(\lambda) \sum_{i=1}^{t_{\lambda}} d_{i}(m_{\lambda}^{t_{\lambda}-(i-1)}-m_{\lambda}^{t_{\lambda}-(i-1)+1}) }\cdot \delta(H).
$$

Next we will obtain a lower bound for $\delta(H)$. In order to achieve this, we introduce two integers, $m_H$ and $h_H$ defined as follows : 
\begin{itemize}
\item let $m_H$ denote the maximal integer $m_H \geq 1$ such that there exists $P,Q \, \in \, H$ of order $\ell^{m_H}$ and such that $\phi_{\ell}(\ell^{m_H -1}P, \ell^{m_H -1}Q)\in \mu_{\ell}$. If such a $m_H$ does not exists, then we put $m_H=0$;
\item let $h_H$ denote the minimal integer such that  in $1\leq h_H\leq t_{\lambda}$ and  $m^{h_H} \leq m_H$. When $m_H=0$ we put $h_H=t_{\lambda}+1$.
\end{itemize}
Recall the inclusions $G_{W_1} \subset ... \subset G_{W_{t_{\lambda}}}$. We can attach to each stabilizer the integer $\delta_{\lambda,i}$ for $1\leq i \leq t_{\lambda}$ and $\lambda|\ell$ which takes value in $\{0,1\}$ depending in the fact that $W_i$ is contained in a maximal isotropic space or not. Since $W_{t_{\lambda}} \subset ... \subset W_1$ we notice that from the moment when one of the $W_i$ is not contained in a maximal isotropic space, we will have $\delta_{\lambda,i}=...=\delta_{\lambda,1}=1$. This can be translated in terms of the integer $h_H$ as follows: 
$$
\delta_{\lambda,1}=...=\delta_{\lambda,t_{\lambda}+1-h_H}=1 \quad \text{and} \quad \delta_{\lambda,t_{\lambda}+1-(h_H -1)}=...=\delta_{\lambda,t_{\lambda}}=0.
$$
Actually the integers $m_H$ and $h_H$ have been introduced in order to have $\delta(H)\gg \ell^{m^{h_H}}$. Note that we can write $m^{h_H}$ in terms of the $\delta_{\lambda,i}$ in the following way:
$$
m^{h_H} = \sum_{i=1}^{t_{\lambda}} (m_{\lambda}^{t_{\lambda}-(i-1)} - m_{\lambda}^{t_{\lambda}-(i-1)+1}) \delta_{\lambda,i}.
$$
Thus, from Lemma \ref{lemma delta}, one can deduce the following inequality :
$$
[K(H):K]\gg \ell^{\sum_{\lambda|\ell} f(\lambda) \sum_{i=1}^{t_{\lambda}} d_{i}(m_{\lambda}^{t_{\lambda}-(i-1)}-m_{\lambda}^{t_{\lambda}-(i-1)+1}) } \ell^{m^{h_H}},
$$
which holds up to some finite index, and hence
$$
[K(H):K]\gg \ell^{\sum_{\lambda|\ell}\sum_{i=1}^{t_{\lambda}}  f(\lambda)(d_i+\delta_{\lambda,i}) (m_{\lambda}^{t_{\lambda}-(i-1)}- m_{\lambda}^{t_{\lambda}-(i-1)+1})}.
$$

We will use some combinatorial arguments, as in \cite[4.2]{HR10}. We have the following equivalence: 

$$
|H| = \ell^{\sum_{\lambda|\ell} \sum_{i=1}^{t_{\lambda}}  f(\lambda)  a_{\lambda,i} m_{\lambda}^i} \ll [K(H):K]^{\gamma(A)} 
\iff
$$
\vspace{0.001cm}
$$
\gamma(A) \geq \max \left\{ \frac{\sum_{\lambda|\ell} \sum_{i=1}^{t_{\lambda}}  f(\lambda)  a_{\lambda,i} m_{\lambda}^i}{\sum_{\lambda|\ell} \sum_{i=1}^{t_{\lambda}}  f(\lambda)(d_i+\delta_{\lambda,i}) (m_{\lambda}^{t_{\lambda}-(i-1)}- m_{\lambda}^{t_{\lambda}-(i-1)+1})} \right\}.
$$
Let us recall that $a_{\lambda,i}=d \alpha_{\lambda,i}$ where $d=2$.
After some modifications with the denominator we obtain
$$
\gamma(A) \geq \max \left\{ \frac{\sum_{\lambda|\ell} \sum_{i=1}^{t_{\lambda}}  f(\lambda)  a_{\lambda,i} m_{\lambda}^i}{\sum_{\lambda|\ell} \sum_{i=1}^{t_{\lambda}}  f(\lambda) m_{\lambda}^i (d_{t_{\lambda}+1-i} + \delta_{\lambda,t_{\lambda}+1-i} - d_{t_{\lambda}+2-i} - \delta_{\lambda,t_{\lambda}+2-i})} \right\}.
$$
where we take the maximum over $m_{\lambda}^{t_{\lambda}}\leq...\leq m_{\lambda}^1$ and we denote the right hand side of the inequality by $M$.
From \cite[Lemma 2.7]{HR10}, one has 
$$
M= \max_{1 \leq k \leq t_{\lambda}}\left\{ \frac{\sum_{\lambda|\ell}\sum_{i=1}^{k} f(\lambda)a_{\lambda,i} }{\sum_{\lambda|\ell}\sum_{i=1}^{k} f(\lambda)(d_{t_{\lambda}+1-i} + \delta_{\lambda,t_{\lambda}+1-i} - d_{t_{\lambda}+2-i} - \delta_{\lambda,t_{\lambda}+2-i})} \right\}.
$$
By convention, one puts $d_{t_{\lambda}+1}=0=\delta_{\lambda,t_{\lambda}+1}$. Let us remark that 
$$r_{t_{\lambda}+1-k}:=\rg_{\Olam} W_{t_{\lambda}+1-k} = \sum_{i=1}^{k} a_{\lambda,i}.$$ 
After some simplifications we obtain:
$$
M= \max_{1 \leq k \leq t_{\lambda}}\left\{ \frac{\sum_{\lambda|\ell}f(\lambda)r_{t_{\lambda}+1-k} }{\sum_{\lambda|\ell}f(\lambda)(d_{t_{\lambda}+1-k}+\delta_{\lambda,t_{\lambda}+1-k})} \right\}.
$$
We can assume, without any loss of generality, that $\delta(H)=~\delta(H_{\lambda'})$ for a fixed place $\lambda'$ over $\ell$. Then, for every $\lambda|\ell$ such that $\lambda \neq \lambda'$, one has $\delta_{\lambda,t_{\lambda}+1-k}=0$ for every $1\leq k\leq t_{\lambda}$ and so,

$$
M= \max_{1 \leq k \leq t_{\lambda}}\left\{ \frac{\sum_{\lambda|\ell}f(\lambda)r_{t_{\lambda}+1-k} }{\delta_{\lambda',t_{\lambda'}+1-k}+\sum_{\lambda|\ell}f(\lambda)d_{t_{\lambda}+1-k}} \right\}.
$$
The maximum will be taken according to the values of $\delta_{\lambda',t_{\lambda}+1-k}$. Two cases can occur:

\begin{itemize}
\item If $1 \leq k < h_H$ then $t_{\lambda'}+1-k \in \llbracket t_{\lambda'}+2-h_H, t_{\lambda'} \rrbracket$ and therefore $\delta_{\lambda',t_{\lambda'}+1-k}=0$.

\item If $h_H \leq k \leq t_{\lambda'}$ then $t_{\lambda'}+1-k \in \llbracket 1, t_{\lambda'}+1-h_{H} \rrbracket$ and therefore $\delta_{\lambda',t_{\lambda'}+1-k}=1$.
\end{itemize}
Then we see that this maximum is
$$
M=\max \left\{ \max_{1\leq k < h_H} \, \frac{\sum_{\lambda|\ell}f(\lambda)r_{t_{\lambda}+1-k} }{\sum_{\lambda|\ell}f(\lambda) d_{t_{\lambda}+1-k}} \, , \, \max_{h_H \leq k < t_{\lambda}} \, \frac{\sum_{\lambda|\ell}f(\lambda)r_{t_{\lambda}+1-k}}{1+\sum_{\lambda|\ell}f(\lambda)d_{t_{\lambda}+1-k}} \right\}.
$$
Let us recall that 
$$
d_{t_{\lambda}+1-k}= \frac{r_{t_{\lambda}+1-k}(4h-1-r_{t_{\lambda}+1-k})}{2}.
$$
We introduce the following functions:
$$f_1 (r_{t_{\lambda}+1-k})=\frac{\sum_{\lambda|\ell}f(\lambda)r_{t_{\lambda}+1-k} }{\sum_{\lambda|\ell}f(\lambda) \frac{r_{t_{\lambda}+1-k}(4h-1-r_{t_{\lambda}+1-k})}{2}} $$
and
$$f_2(r_{t_{\lambda}+1-k})=\frac{\sum_{\lambda|\ell}f(\lambda)r_{t_{\lambda}+1-k}}{1+\sum_{\lambda|\ell}f(\lambda)\frac{r_{t_{\lambda}+1-k}(4h-1-r_{t_{\lambda}+1-k})}{2}}. $$ 
Then we have:
$$
M=\max \left\{ \max_{1\leq k < h_H} \, f_1(r_{t_{\lambda}+1-k}) \, , \, \max_{h_H \leq k < t_{\lambda}} \, f_2(r_{t_{\lambda}+1-k}) \right\}.
$$
An easy study shows that the functions $f_1$ and $f_2$ are increasing in their domains of definition. Note that their maxima are obtained at $r_{t_{\lambda}+1-k}=h$ and $r_{t_{\lambda}+1-k}=2h$ respectively. When one of the $W_i$ is contained in a maximal isotropic space we remark that its rank is at most $h$. The maximal case for $f_2$ occurs when $W_1=\Tl(A)$. Then
$$
M=\max \left\{ \frac{2}{3h-1} \, , \, \frac{2eh}{eh(2h-1)+1} \right\}.
$$
Note that the case $H\subset \Alinf$ can be reduced to the case $H \subset \Al$. We conclude that:
$$
\gamma(A)=\frac{2(2eh)}{eh(2h-1)+1}=\frac{2\dim\,A}{\dim \, \Hg+1},
$$
as expected.

Finally, in order to conclude the proof, one needs to prove this theorem in the case where the prime number $\ell$ is in the finite set $\mathcal{S}$. In order to achieve this, one can follow a discussion of different cases as in \cite[Paragraph 8]{HR}.
\end{proof}

Finally, we generalize Conjecture \ref{conjecture HR} to the case where $A$ is isogenous to a product of simple abelian varieties of type I, II or III and fully of Lefschetz type. More precisely we present a new result where only the type IV, in the sense of Albert's classification, is excluded.\\

\begin{thm}\label{theo3bis}
Let $A$ be an abelian variety defined over a number field $K$ isogenous over $\overline{K}$ to $\prod_{i=i}^{d}A_i^{n_i}$ where the abelian varieties $A_i$ are pairwise non isogenous over $\overline{K}$ of dimension $g_i$ and $n_i\geq 1$. We suppose that the abelian varieties $A_i$ are simple, not of type IV and fully of Lefschetz type. 
For every non empty subset $I\subset \{1,...,d\}$ we denote $A_I:=\prod_{i\in I} A_i$. Moreover we define $e_i=[E_i:\Q]$ where $E_i=Z(\Endzero(A_i))$, $h_i= \dimrel A_i$  and
$$
d_i=\left\{
\begin{aligned}
&1 \text{ if }A_i \text{ is of type I,}\\
&2 \text{ if }A_i \text{ is of type II or III,}
\end{aligned}\right.
\qquad
\eta_i=\left\{
\begin{aligned}
&1 \text{ if }A_i \text{ is of type I or II,}\\
&-1 \text{ if }A_i \text{ is of type III.}
\end{aligned}\right.
$$
Then one has:
$$
\gamma(A)=\max_{I} \left\{ \frac{2\sum_{i\in I} n_i d_i e_i h_i}{1 + \sum_{i\in I} e_i (2 h_i^2 +\eta_i h_i)}\right\}=\max_{I}\left\{ \frac{2\sum_{i\in I} n_i\dim A_i}{1 + \dim \rm Hg(\prod_{i\in I} A_i)}\right\}.
$$
\end{thm}

\begin{proof}
In order to prove Theorem \ref{theo3bis} one needs to show that a product of simple abelian varieties which are fully of Lefschetz type is an abelian variety fully of Lefschetz type.
In order to accomplish this we use \cite[Theorem 1A]{Ich} and \cite[Theorem 4.1]{LomHl}. Let $A_i$ and $A_j$ be two simple abelian varieties not isogenous. Assume that $A_i$ and $A_j$ are not of type IV and that they are fully of Lefschetz type. Then we have:
\begin{enumerate}
\item (Ichikawa) $\rm{Hg}(A_i\times A_j)=\rm{Hg}(A_i)\times \rm{Hg}(A_j)$,
\item (Lombardo) $\Hl(A_i\times A_j)=\Hl(A_i)\times \Hl(A_j)$ for every prime number $\ell$.
\end{enumerate}
Hence, using \cite[Theorem 1.14]{HR} and the same techniques developed in the proof of Theorem \ref{theo1} and the proof of \cite[Theorem 1.14]{HR}, we conclude the proof of Theorem \ref{theo3bis}.
\end{proof}

\section{Order of the extension generated by a torsion point}\label{cinq}

The following results are consequences of the method of proof of Theorems \ref{theo1} and \ref{theo3}.

\begin{thm}\label{th:5.1}
Let $A$ be an abelian variety defined over a number field $K$, simple of type III, of relative dimension $h$ and fully of Lefschetz type. There exists a constant $c_1:=c_1(A,K)>0$ such that, for every torsion point $P$ of order $m$ in $A(\overline{K})$, one has:
$$
[K(P):K]\geq c_1^{\omega(m)}m^{2h},
$$
where $\omega(m)$ is the number of prime factors of $m$.
\end{thm}

\begin{thm}\label{th:5.2}
Let $A$ be an abelian variety defined over a number field $K$, isogenous over $\overline{K}$ to $\prod_{i=i}^{d}A_i^{n_i}$ where the abelian varieties $A_i$ are pairwise non isogenous over $\overline{K}$. We suppose that the abelian varieties $A_i$ are simple, not of type IV, of relative dimension $h_i$ and fully of Lefschetz type. There exists a constant $c_1:=c_1(A,K)>0$ such that, for every torsion point $P$ of order $m$ in $A(\overline{K})$, one has:
$$
[K(P):K]\geq c_1^{\omega(m)}m^{2h},
$$
where $h=\min_{i} h_i$.
\end{thm}

\begin{proof}
Since Theorem \ref{th:5.2} follows from Theorem \ref{th:5.1}, we prove the latter and assume that the abelian variety $A$ is simple. Let $P$ be a torsion point of order $m$. Then $P=P_1+...+P_r$ where each $P_i$ is a point of order $\ell_i^{n_i}$ and $m=\prod_{i=1}^r \ell_i^{n_i}$, for some positive integers $n-i$. Denoting by $H_P$ the $\End(A)$-module generated by $P$, one has $K(P)=K(H_P)$. By the independence criterion of $\ell$-adic representations introduced by Serre, we know that, up to some multiplicative constants, uniform in $m$ and in $P$, we have the following inequality:
$$
[K(P):K]=[K(P_1,..., P_r):K]\gg \prod_{i=1}^r [K(P_i):K].
$$
Moreover, one has the following inequality, up to some multiplicatives constants, uniform in $\ell_i$ and in $P_i$:
$$
[K(P_i):K]\gg \ell_i^{2hn_i}.
$$
Let $\omega(m)$ be the number of prime factors of $m$. There exists a positive constant $c_1:=c_1(A,K)$ such that  
\begin{equation}\label{numerote}
[K(P):K]=[K(P_1,..., P_r):K]\gg \prod_{i=1}^r [K(P_i):K]\gg \prod_{i=1}^r c_1\ell_i^{2hn_i}\geq c_1^{\omega(m)}m^{2h}.
\end{equation}classical estimate
\end{proof}

\begin{rmk}
Using the following classical estimate (see \cite[Chapter I.5]{Ten})
$$
\omega(m)\leq c\cdot \frac{\log m}{\log\log m},
$$
where $c>0$ is an absolute constant, one can rewrite inequality \eqref{numerote}:
\begin{equation}
[K(P):K]\geq m^{2h-\frac{c \log c_1}{\log \log m}}.
\end{equation}
where $c_1$ is a constant that only depends on $A$.
\end{rmk}

\section{New results concerning the Mumford--Tate Conjecture}\label{section 5}

Recall that one of the main hypotheses of Theorems \ref{theo1} and \ref{theo3} is that the abelian variety $A$ must be fully of Lefschetz type. In particular, the Mumford--Tate Conjecture must hold for $A$. In this direction, the following theorems give examples of abelian varieties that are fully of Lefschetz type and for which Theorems \ref{theo1} and \ref{theo3} are unconditional. 

Some results of Noot and Shimura show that there exist abelian varieties, defined over a number field K, of type III in the sense of Albert's classification, such that the Mumford--Tate Conjecture holds for them. More precisely, one can state the following theorem:

\begin{thm}[Noot, Shimura]
Let $h$ be an integer $h\geq 3$, $E$ a totally real field and $D$ a quaternion algebra totally definite over $E$. Then
\begin{enumerate}
\item There exists a moduli space (called Shimura variety) which parametrizes abelian varieties whose endomorphism algebra contains $D$. Moreover, if $A$ corresponds to a general point of the Shimura variety, then $\Endzero(A)=D$ and $\MT=\mathcal{L}(A)$  (Lefschetz group of $A$).
\item There exists a number field $K$ and an abelian variety $A$, defined over $K$, such that $\Endzero(A)=D$, $\MT=\mathcal{L}(A)$ and the Mumford--Tate Conjecture holds for $A$.
\end{enumerate}
\end{thm}
\begin{rmk}
The first point is due to Shimura (see \cite[Theorem 9.1]{lange}) and the second point is the \cite[Th\'eor\`eme 1.7]{Noot}.
\end{rmk}

The following theorem gives a correction of a subtle point in \cite[Theorem 5.11]{BGK}:

\begin{thm}\label{theo BGK}
Let $A$ be a simple abelian variety of dimension $g$ and of type III. Let us recall that $g=2eh$ where $e=[E:\Q]$ and $h$ is the relative dimension. We assume that  
\begin{enumerate}
\item $h\in \{2k+1, k \in \N\}\setminus \left\{\frac{1}{2}\binom{2^{m+2}}{2^{m+1}}, m\in \N\right\}$.
\end{enumerate}
Then $A$ is fully of Lefschetz type.
\end{thm}

\begin{rmk}
Theorem \ref{theo BGK} is stated in \cite{BGK} with the hypothesis $h$ odd, nevertheless, it seems to be important to exclude the values of the form $\frac{1}{2}\binom{2^{m+2}}{2^{m+1}}$ with $m\in \N$, as pointed out by \cite[Remark 2.27]{LomHl}

The proof of Theorem \ref{theo BGK} follows closely the one from \cite[Theorem 5.11]{BGK} and we only gives details when the proof departs from \cite[Theorem 5.11]{BGK}. Nevertheless it is important to point out that the fact that we need to assume that the relative dimension $h$ does not belong to $\left\{\frac{1}{2}\binom{2^{m+2}}{2^{m+1}}, m\in \N\right\}$ has not been taken into consideration in \cite[Lemma 4.13]{BGK}.
\end{rmk}

With the notations of \cite[Paragraph 4]{BGK} we consider an ideal $\lambda$ in $\Oe$ such that $\lambda|\ell$. Let us introduce the following $\lambda$-adic representation:

$$
\rho_{\lambda}^{\circ}:G_K \to \GL(\Vlam).
$$
We denote by $G_{\lambda}$ the Zariski closure of $\rho_{\lambda}^{\circ}(G_K)$ and $\mathfrak{g}_{\lambda}$ the Lie algebra of $G_{\lambda}$. Let $\phi_{\lambda^{\infty}}^{\circ}$ be the $\lambda$-adic pairing defined over $\Vlam$, with notations borrowed from \cite{BGK}. We can therefore state the following lemma:

\begin{lemma}(\cite[Lemma 4.13]{BGK})\label{lemme413}
Under the hypothesis of Theorem \ref{theo BGK} we have the following equality:
$$
\mathfrak{g}_{\lambda}^{ss}=\mathfrak{so}_{\Vlam,\phi_{\lambda^{\infty}}^{\circ}}.
$$
\end{lemma}
 
The proof of Lemma \ref{lemme413} is almost the same as the one on pages 175-176 of \cite{BGK}. Let us recall some notations. Let  $\overline{\Vlam}:=\Vlam\otimes \Ql$ be a Lie module for which we have the following decomposition:

\begin{equation}\label{pt}
\overline{\Vlam}=E(\omega_1)\otimes ... \otimes E(\omega_t),
\end{equation}
where for every $1\leq i \leq t $, $E(\omega_i)$ is an irreducible module of the Lie algebra of a maximal weight $\omega_i$.

\begin{figure}[!h]
\begin{center}

\begin{tabular}{|c|c|c|c|}
   \hline 
   Root system & Minuscule weight & Dimension & Duality properties \\
   \hline 
   \multirow{2}{*}{$A_\ell \, (\ell\geq1)$} & \multirow{2}{*}{$\omega_r, \, 1\leq r\leq \ell$} & \multirow{2}{*}{$\binom{\ell+1}{r}$} & $(-1)^r$, if $r=\frac{\ell+1}{2}$\\ 
   & & &  $0$ otherwise\\ 
   \hline 
   \multirow{2}{*}{$B_\ell \, (\ell\geq2)$} & \multirow{2}{*}{$\omega_{\ell}$} & \multirow{2}{*}{$2^{\ell}$} & $+1$, if $\ell \equiv 3,0 \mod 4$\\ 
   & & &  $-1$, if $\ell \equiv 1,2 \mod 4$\\ 
   \hline 
   $C_\ell \, (\ell\geq2)$ & $\omega_1$ & $2\ell$ & $-1$\\
   \hline 
   \multirow{4}{*}{$D_\ell \, (\ell\geq3)$} & $\omega_1$ & $2\ell$ & $+1$\\
   \cline{2-4}
   &    \multirow{3}{*}{$\omega_{\ell-1}, \, \omega_{\ell}$} &    \multirow{3}{*}{$2^{\ell-1}$} & $+1$, if $\ell \equiv 0 \mod 4$ \\
   &&& $-1$, if $\ell \equiv 2 \mod 4$\\
   &&& $0$, if $\ell \equiv 1 \mod 2$\\
   
   \hline
\end{tabular}
\caption{Minuscule weight for the classical Lie algebras}\label{table minuscule}
\end{center}
\end{figure}

The penultimate sentence of the proof of Lemma \ref{lemme413} states that, since h is odd, the investigation of the tables of minuscule weights (see figure \ref{table minuscule} or \cite[Table 1 et 2, P. 213--214]{bourbaki}) and the dimensions of associated representations shows that the tensor product \eqref{pt} can contain only one factor which is orthogonal. Therefore one has two possibilities:
\begin{enumerate}
\item either of type $D_n$, minuscule weight $w_1$ and dimension $2n$,

\item or type $A_{4k+3}$, minuscule weight $w_{2k+2}$ and dimension $\binom{4k+4}{2k+2}$,
\end{enumerate}
where $k\in \N$, such that $\ell=4k+3$ and $r=2k+2$ in figure \ref{table minuscule}.

It seems that the last possibility has been overlooked in \cite{BGK} when the authors defined the class $\mathcal{B}$ of abelian varieties. Let us discuss the last point. 

The representation $A_{4k+3}$ is orthogonal indeed because $r=2k+2$ is an even number. Moreover we know that $\binom{4k+4}{2k+2}=2h$, because there is only one orthogonal representation. We settle in the next lemma the case in which values of the integer $k$ the number $\frac{1}{2}\binom{4k+4}{2k+2}$ is odd.

\begin{lemma}
The integer $\frac{1}{2}\binom{4k+4}{2k+2}$ is odd if and only if $k+1=2^m$ for an integer $m\geq1$.
\end{lemma}

\begin{proof}
First one needs to study the parity of the integer $\binom{4k+4}{2k+2}$. More precisely, one needs to compute its $2$-adic valuation. One has the following equality:
\begin{equation}\label{egalité}
\binom{4k+4}{2k+2}=\frac{2^{k+1}}{(k+1)!}\times \text{(odd number)}.
\end{equation}
The equality above shows that the $2$-adic valuation of $\binom{4k+4}{2k+2}$ is related to the $2$-adic valuation of $(k+1)!$.
Let us consider the binary development of $k+1=\sum_{i=0}^m \varepsilon_i 2^i$, with $m\geq1$ and $\varepsilon_m=1$.
One knows that $v_2((k+1)!)=\sum_{h\geq1} \lfloor \frac{k+1}{2^h} \rfloor$.
From equality \eqref{egalité}, we find that
$$
v_2(\binom{4k+4}{2k+2})=k+1-v_2((k+1)!)=\varepsilon_0+...+\varepsilon_m=\varepsilon_0+...+\varepsilon_{m-1}+1.
$$
Thus
$$
\frac{1}{2}\binom{4k+4}{2k+2} \text{is odd} \iff \varepsilon_0=...=\varepsilon_{m-1}=0 \iff k+1=2^m \iff k=2^m-1.
$$
\end{proof}


The main goal of the rest of the paragraph is to extend Pink's results in order to prove new cases of the Mumford--Tate Conjecture. Concerning this last conjecture, Pink applies his results to the case where $\End(A)=\Z$, nevertheless, as he pointed out (\cite[Remark p. 33]{Pink98}) his results can be generalized. Guided by this remark, we obtained new results on simple abelian varieties of type III, which we are going to present. More precisely we are going to state an analogous result to \cite[Proposition 4.7]{Pink98} in the case of absolute irreducible representations which are orthogonal.

Let $A$ be a simple abelian variety of type III defined over a number field $K$ of dimension $\dim A=g$. Thanks to \cite[Theorem 3.23]{BGK}, we know that the associated representation of the vector space $\Vlam:=\Tlam\otimes \Elam$ is absolutely irreducible and that $\Vl = \bigoplus_{\lambda|\ell} \Vlam \oplus \Vlam$. As in the previous section, we consider the $\lambda$-adic representation $\rho_{\lambda}^{\circ}$, the algebraic group $\Glam$ together with his associated Lie algebra $\mathfrak{g}_{\lambda}$. 

In order to state an analogue of \cite[Proposition 4.7]{Pink98}, we need to define Mumford--Tate pairs (see \cite[Definition 4.1]{Pink98}). Let $G$ be a reductive algebraic group and $\rho$ a faithful finite dimensional representation of $G$.

\begin{defn}
The pair $(G,\rho)$ is a weak Mumford--Tate pair of weight $\{0,1\}$, if and only if there exist cocharacters $\mu_i:\G_{m,\overline{K}}\to G_{\overline{K}}$ for every $1\leq i\leq k $ such that
\begin{itemize}
\item $G_{\overline{K}}$ is generated by the images of all $G(\overline{K})$--conjugates of all $\mu_i$, and
\item the weights of each $\rho\circ\mu_i$ are in $\{0,1\}$.
\end{itemize}
The pair $(G,\rho)$ is a strong Mumford--Tate pair of weight $\{0,1\}$, if and only if the pair is a weak Mumford--Tate pair and the cocharacters are conjugate under the action of $\Gal(\overline{K}/K)$.
\end{defn} 
   
We assume that the center of the endomorphism algebra $E=Z(\Endzero(A))$ is equal to $\Q$. Then, $[E:\Q]=1$ and $g=2eh=2h$. Thanks to \cite[Fact 5.9]{Pink98} we know that, in this case, all weak Mumford--Tate pairs are actually strong Mumford--Tate pairs. 
   
Let us recall that for a reductive algebraic group $G$ we have $G=\G_{m,\overline{K}}\cdot G^{der}$ where $G^{der}$ is the derived group of $G$. The representation $\rho$ can be decomposed into the exterior tensor product:
$$
\rho \cong \rho_0 \otimes \rho_1 \otimes ... \otimes \rho_s,
$$   
where every $\rho_i$ is an absolute irreducible representation, $\rho_0$ is the standard representation associated to $\G_{m,\overline{K}}$ and the $\rho_i$ is the representation associated to the factor $G_i$, for $1\leq i\leq s$.
   
\begin{prop}(\cite[Proposition 4.4]{Pink98})\label{prop 4.4}
If $(G,\rho)$ is a strong Mumford--Tate pair of weight $\{0,1\}$ over $K$ such that the representation $\rho$ is absolutely irreducible, then $G^{der}$ is either almost simple over $K$ or trivial. In particular, all simple tensor factors of $(G,\rho)$ over $\overline{K}$ have the same type as in figure \ref{table minuscule} with the same integer $\ell$.
\end{prop}

From figure \ref{table minuscule} one obtains the following table of orthogonal representations:
\begin{figure}[!h]
\begin{center}

\begin{tabular}{|c|c|c|c|}
   \hline 
   Root system & Minuscule weight & Dimension & Duality properties \\
   \hline 
   $A_\ell \, (\ell\geq1)$ & $\omega_r, \, 1\leq r\leq \ell$ & $\binom{\ell+1}{r}$ & $+1$ with $r=\frac{\ell+1}{2}\equiv 0 \mod 2$\\ 
   \hline 
   $B_\ell \, (\ell\geq2)$ & $\omega_{\ell}$ & $2^{\ell}$ & $+1$, if $\ell \equiv 3,0 \mod 4$\\ 
   \hline 
   \multirow{2}{*}{$D_\ell \, (\ell\geq3)$} & $\omega_1$ & $2\ell$ & $+1$\\
   \cline{2-4}
   &  $\omega_{\ell-1}, \, \omega_{\ell}$ &   $2^{\ell-1}$ & $+1$, if $\ell \equiv 0 \mod 4$ \\   
   \hline
\end{tabular}
\caption{Minuscule weight for orthogonal Lie algebras}\label{table orthogonale}
\end{center}
\end{figure}

Let us recall that we denote by $s$ the number of simple factors of $\rho_{|G^{der}}$. The following proposition is analogous to \cite[Proposition 4.5]{Pink98}:

\begin{prop}
Consider a strong Mumford--Tate pair $(G,\rho)$ of weight $\{0,1\}$ over $K$ such that the representation $\rho_{|G^{der}}$ is absolutely irreducible and orthogonal. Then all simple tensor factors of $(G,\rho)$ over $\overline{K}$ are either
\begin{itemize}
\item orthogonal, for an arbitrary $s\geq 1$;
\item or symplectic, when $s$ is even.
\end{itemize}
\end{prop}

\begin{proof}
Since $\rho_{|G^{der}}$ is absolutely irreducible and orthogonal, and thanks to Proposition \ref{prop 4.4}, we know that all simple factors of the tensor product of $(G,\rho)$ over $\overline{K}$ have the same type. One has two possible cases. Either $s$ is arbitrary and all simple factors are orthogonal, or $s$ is even, and all simple factors are symplectic. Recalling that a tensor product of two symplectic representations is orthogonal, which conclude the proof.
\end{proof}

We can therefore state an analogous proposition of \cite[Proposition 4.7]{Pink98}:

\begin{prop}\label{prop 6.10}
Consider a strong Mumford--Tate pair $(G,\rho)$ of weights $\{0,1\}$ over $K$ such that $\rho_{|G^{der}}$ is absolutely irreducible and orthogonal. Assume that $n:=\dim \rho$ is greater than 1 and $n$ is neither:
\begin{itemize}
\item a power of $2$ nor;
\item $(2a)^k$ with $k$ an even integer, $k\geq 2$ and $a\geq2$ nor;
\item a power of $\binom{2k}{k}$ where $k$ is an even integer $k\geq 2$ nor;
\item an even power of $\binom{2k}{k}$ where $k$ is an odd integer $k\geq 1$.
\end{itemize}
Then we have $G=\Gm \cdot SO_{n,\Q}$.
\end{prop}

\begin{rmk}
Actually we prove a slightly stronger result than the stated above. For details see the proof, in particular \eqref{eq:set}. Let us remark that the set of dimensions excluded in Proposition \ref{prop 6.10} is larger and contains the set $\Sigma'$ defined in \eqref{eq:set}.
\end{rmk}

\begin{proof}
Let $T_{\ell}\in \{A_{\ell}, B_{\ell}, C_{\ell}, D_{\ell}\}$ be one of the four simple types in figure \ref{table minuscule}. The main goal of this proof is to identify all possible values of $h$ for which the representation is not orthogonal. In order to do this, we study every possible combination that can occur. Let  $n=2h=\dim \rho$.

The \textbf{first case} we consider is when the representation associated to the type $T_{\ell}$ is orthogonal. Hence we want to obtain the type $D_k$ associated to the standard orthogonal representation. In order to achieve this, we have to remove the four remaining cases. Let us consider a more detailed version of figure \ref{table orthogonale}, (see figure \ref{orthogonales}).

\begin{figure}[!h]
\begin{center}
\begin{tabular}{|c|c|c|c|}
   \hline 
   Root system & Minuscule weight & Dimension & Duality properties \\
   \hline 
   $A_{4k+3} \, (k\geq0)$ & $\omega_{1}, ..., \omega_{4k+3}$ & $\binom{4k+4}{2k+2}$ & $+1$\\ 
   \hline 
   $B_{4k} \, (k\geq1)$ & $\omega_{4k}$ & $2^{4k}$ & $+1$\\ 
   \hline 
   $B_{4k+3} \, (k\geq0)$ & $\omega_{4k+3}$ & $2^{4k+3}$ & $+1$\\ 
   \hline 
   {\color{red}$D_k \, (k\geq3)$} & {\color{red}$\omega_1$} & {\color{red}$2k$} & {\color{red}$+1$}\\
   \hline 
   $D_{4k} \, (k\geq1)$&  $\omega_{4k-1}, \, \omega_{4k}$ &   $2^{4k-1}$ & $+1$ \\   
   \hline
\end{tabular}
\caption{Minuscule weights : orthogonal}\label{orthogonales}
\end{center}
\end{figure}
We present here the list of the four remaining cases that need to be excluded:
\begin{itemize}
\item \underline{Type $D_{4k}$} of dimension $2^{(4k-1)s}$ where $k\geq 1$ and $s\geq1$, therefore for every $k\geq 1$ and $s\geq1$
$$h\neq 2^{(4k-1)s-1}.$$ 

\item \underline{Type $B_{4k+3}$} of dimension $2^{(4k+3)s}$ where $k\geq 0$ and $s\geq1$, hence for every $k\geq 0$ and $s\geq1$
$$h\neq 2^{(4k+3)s-1}.$$

\item \underline{Type $B_{4k+3}$} of dimension $2^{(4k)s}$ where $k\geq 1$ and $s\geq1$, thus for every $k\geq 1$ and $s\geq1$
$$h\neq 2^{4ks-1}.$$ 

\item \underline{Type $A_{4k+3}$} of dimension $(\binom{4k+4}{2k+2})^s$ where $k\geq 0$ and $s\geq1$, therefore for every  $k\geq 0$ and $s\geq1$
$$h\neq \frac{1}{2}(\binom{4k+4}{2k+2})^s.$$
\end{itemize} 
Let us remark that the two first cases give the same conditions for the integer $h$. We can therefore give the following set:
\begin{equation}
\Sigma_1=\left\{2^{(4k+3)s-1}, \, k\geq 0, \, s\geq1\right\}\bigcup\left\{2^{4ks-1}, \, k\geq 1, \, s\geq1\right\}\bigcup\left\{\frac{1}{2}\big(\binom{4k+4}{2k+2}\big)^s, \, k\geq 0, \, s\geq1\right\}.
\end{equation}
The previous set $\Sigma_1$ contains all the possible integers that correspond to the dimension of the representations that are not orthogonal.

The \textbf{second case} we consider is the one where the type $T_{\ell}$ is associated to a symplectic representation and the number of such representations is an even number. Let us consider the following figure which corresponds to the symplectic representations (see figure \ref{symplectiques}):

\begin{figure}[!h]
\begin{center}
\begin{tabular}{|c|c|c|c|}
   \hline 
   Root system & Minuscule weight & Dimension & Duality properties \\
   \hline 
   $A_{4k+1} \, (k\geq0)$ & $\omega_{1}, ..., \omega_{4k+1}$ & $\binom{4k+2}{2k+1}$ & $-1$\\ 
   \hline 
   $B_{4k+1} \, (k\geq1)$ & $\omega_{4k+1}$ & $2^{4k+1}$ & $-1$\\ 
   \hline 
   $B_{4k+2} \, (k\geq0)$ & $\omega_{4k+2}$ & $2^{4k+2}$ & $-1$\\ 
   \hline 
   $C_k \, (k\geq2)$ & $\omega_1$ & $2k$ & $-1$\\
   \hline 
   $D_{4k+2}  \, (k\geq1)$&  $\omega_{4k+1}, \, \omega_{4k+2}$ &   $2^{4k+1}$ & $-1$ \\   
   \hline
\end{tabular}
\caption{Minuscule weights: symplectic}\label{symplectiques}
\end{center}
\end{figure}
As in the previous case, we presente the list of the five remaining cases that need to be excluded. 
\begin{itemize}
\item \underline{Type $D_{4k+2}$} of dimension $2^{(4k+1)2s}$ where $k\geq 1$ and $s\geq1$, therefore for every $k\geq 1$ and $s\geq1$ 
$$h\neq 2^{2s(4k+1)-1}.$$

\item \underline{Type $C_{k}$} of dimension $(2k)^{2s}$ where $k\geq 2$ and $s\geq1$, hence for every $k\geq 2$ and $s\geq1$
$$h\neq 2^{2s-1}k^{2s}.$$

\item \underline{Type $B_{4k+2}$} of dimension $2^{(4k+2)2s}=2^{4s(2k+1)}$ where $k\geq 0$ and $s\geq1$, thus for every $\alpha\geq 1$ odd and $s\geq1$
$$h\neq 2^{4s\alpha-1}.$$

\item \underline{Type $B_{4k+1}$} of dimension $2^{(4k+1)2s}$ where $k\geq 1$ and $s\geq1$, therefore for every $k\geq 1$ and $s\geq1$ $$h\neq 2^{2s(4k+1)-1}.$$

\item \underline{Type $A_{4k+1}$} of dimension $(\binom{4k+2}{2k+1})^{2s}$ where $k\geq 0$ and $s\geq1$, hence for every $k\geq 0$ and $s\geq1$
$$h\neq \frac{1}{2}(\binom{4k+2}{2k+1})^{2s}.$$

\end{itemize} 

Let us remark that the first and the fourth cases give the same conditions for the integer $h$. We define the following set of numbers:
\begin{equation}
\begin{aligned}
\Sigma_2=& \{2^{2s(4k+1)-1}, \, k\geq 1, \, s\geq1\}\bigcup\{2^{4\alpha s-1}, \, \alpha \, \text{odd}, \, s\geq1\}\\
& \bigcup\{2^{2s-1}k^{2s}, \, k\geq 2, \, s\geq1\}\bigcup\left\{\frac{1}{2}\big(\binom{4k+2}{2k+1}\big)^{2s}, \, k\geq 0, \, s\geq1\right\}.
\end{aligned}
\end{equation}
The previous set $\Sigma_2$ contains all the possible integers that correspond to the dimension of the representations that are not orthogonal.

Let us concluded by saying that if $h\notin \Sigma_1 \bigcup \Sigma_2$ the representation $\rho_{|G^{der}}$ is absolutely irreducible, orthogonal and corresponds to the standard representation, where $\Sigma_1$ and $\Sigma_2$ are the sets of numbers which were defined above.

\end{proof}

Let us consider the following set:
\begin{equation}
\Sigma':=\Sigma_1\bigcup\Sigma_2.
\end{equation}

After simplifications one obtains as in \eqref{sigmaset}:
\begin{equation}\label{eq:set}
\begin{aligned}
\Sigma':=\bigcup_{s>1} & \left\{ \left\{2^{(4k+3)s-1}, \, k\geq 0\right\} \cup \left\{2^{4ks-1}, \, k\geq 1\right\} \cup \left\{2^{2s(4k+1)-1}, \, k\geq 1\right\} \cup \left\{2^{2s-1}k^{2s}, \, k\geq 2\right\}\right.\\
&\left. \cup \left\{\frac{1}{2}\big(\binom{4k+4}{2k+2}\big)^s, \, k\geq 0\right\} \cup \left\{\frac{1}{2}\big(\binom{4k+2}{2k+1}\big)^{2s}, \, k\geq 0\right\}\right\}.
\end{aligned}
\end{equation}

Therefore we can state the following theorem which generalizes results of \cite{Pink98}:

\begin{thm}\label{theo Pink}
Let $A$ be a simple abelian variety of type III such that the center of $\Endzero(A)$ is $\Q$. Assume that the relative dimension $h$ is not in the set $\Sigma'$. Then, the abelian variety is fully of Lefschetz type. In particular, the Mumford--Tate Conjecture holds for the abelian variety $A$.
\end{thm}

Using Theorems \ref{theo BGK}, \ref{theo Pink}, some results of Ichikawa  \cite{Ich} and Lombardo \cite[Th\'eor\`eme 4.1]{LomHl} and \cite[Corollaires 1.8 and 1.15]{HR} one can prove that the product of simple abelian varieties of type I, II or III that are fully of Lefschetz type is also of Lefschetz type. This last results prove Corollaries \ref{corollaire A simple} and \ref{corollaire A produit}.\\

\textbf{Acknowledgments} I thank Marc Hindry for introducing this problem to me and for the enlightening discussions we had. I also thank Yunqing Tang and Davide Lombardo for helpful discussions and remarks on previous drafts of this paper. Finally I thank the anonymous referee for her/his comments and suggestions on both the mathematics and the writing of this paper.
%


\end{document}